\documentclass[11pt]{amsart}
\usepackage{latexsym}
\usepackage{amsmath,amssymb,amsfonts}
\usepackage{graphics}
\usepackage[all]{xy}

\def\eea{\end{eqnarray*}}

\newtheorem{thm}{Theorem}[section]
\newtheorem{prop}[thm]{Proposition}

\newtheorem{lemma}[thm]{Lemma}

\newtheorem{remark}[thm]{Remark}

\begin{document}
\renewcommand{\theequation}{\thesection.\arabic{equation}}

\title[4-d static spaces and related spaces with harmonic curvature]{Four dimensional static and related critical spaces with harmonic curvature}

\author{Jongsu Kim and Jinwoo Shin}
\date{\today}

\address{Dept. of Mathematics, Sogang University, Seoul, Korea}
\email{jskim@sogang.ac.kr}

\address{Dept. of Mathematics, Sogang University, Seoul, Korea}
\email{shinjin@sogang.ac.kr}

\thanks{This work was supported by the National Research Foundation of Korea(NRF) grant funded by the Korea government(MOE), for the first author (NRF-2010-0011704)}

\keywords{static space, harmonic curvature, Codazzi tensor, critical point metric}

\subjclass[2010]{53C21, 53C25}

\begin{abstract}
In this article we study any 4-dimensional Riemannian manifold $(M,g)$ with harmonic curvature which admits a smooth nonzero solution $f$ to the following equation
\begin{eqnarray} \label{0002bx}
\nabla df = f(Rc -\frac{R}{n-1} g) + x Rc+ y(R) g.
\end{eqnarray}
where $Rc$ is the Ricci tensor of $g$, $x$ is a constant and $y(R)$ a function of the scalar curvature $R$. We show that
a neighborhood of any point in some open dense subset of $M$ is locally isometric to one of the following five types;
{\rm (i)} $ \mathbb{S}^2(\frac{R}{6}) \times   \mathbb{S}^2(\frac{R}{3})$ with $R>0$,
{\rm (ii)} $ \mathbb{H}^2(\frac{R}{6}) \times   \mathbb{H}^2(\frac{R}{3})  $ with $R<0$, where $\mathbb{S}^2(k) $ and  $\mathbb{H}^2(k) $ are the two-dimensional Riemannian manifold with constant sectional curvature $k>0$ and $k<0$, respectively,
{\rm (iii)}  the static spaces in Example 3 below,
{\rm (iv)}  conformally flat static spaces described in Kobayashi's \cite{Ko}, and
{\rm (v)} a Ricci flat metric.

We then get a number of Corollaries, including the classification of the following four dimensional spaces with harmonic curvature;  static spaces, Miao-Tam critical metrics and  $V$-static spaces.

The proof is based on the argument from a preceding study of gradient Ricci solitons \cite{Ki}. Some Codazzi-tensor properties of Ricci tensor, which come from the harmonicity of curvature, are effectively used.

\end{abstract}

\maketitle

\setcounter{section}{0}
\setcounter{equation}{0}

\section{Introduction}

In this article we consider an $n$-dimensional Riemannian manifold $(M, g)$ with constant scalar curvature $R$ which admits a smooth nonzero solution $f$ to the following equation
\begin{eqnarray} \label{0002bx}
\nabla df = f(Rc -\frac{R}{n-1} g) + x \cdot Rc+ y(R) g.
\end{eqnarray}
where $Rc$ is  the Ricci curvature of $g$, $x$ is a constant and $y(R)$ a function of $R$.
There are several well-known classes of spaces which admit such solutions. Below we describe them and briefly explain their geometric significances and recent develpoments.

\medskip
A {\it static space} admits by definition a smooth nonzero solution $f$ to
\begin{eqnarray} \label{0002bx1}
\nabla df = f(Rc -\frac{R}{n-1} g).
\end{eqnarray}

A Riemannian geometric interest of a static space comes from the fact that the scalar curvature functional ${\frak S}$, defined on the space ${\frak M}$ of smooth Riemannian metrics on a closed manifold, is  locally surjective at  $g \in {\frak M} $ if there is no nonzero smooth function satisfying  (\ref{0002bx1}); see Chapter 4 of \cite{Be}.

This interpretation also holds in local sense.
 Roughly speaking, if no nonzero smooth function on a compactly contained subdomain $\Omega$ of a smooth manifold satisfies  (\ref{0002bx1}) for a Riemannian metric $g$ on $\Omega$, then the scalar curvature functional defined on the space of
Riemannian metrics on $\Omega$ is locally surjective at  $g$ in a natural sense; see Theorem 1 of Corvino \cite{Co}. This local viewpoint has been developed to make remarkable progress in Riemannian and Lorentzian geometry \cite{CIP, Co, CEM, CS}.

Kobayashi \cite{Ko} studied a classification of conformally flat static spaces. In his study the list of {\it complete} ones is made. Moreover, all {\it local} ones are described for all varying parameter conditions and initial values of the static space equation. Indeed, they belong to the cases {\rm I}$\sim${\rm VI} in the section 2 of \cite{Ko} and the existence of solutions in each case is thoroughly discussed.
 Lafontaine \cite{La} independently proved a classification of closed conformally flat static space.
 Qing-Yuan \cite{QY} classified complete Bach-flat static spaces which contain compact level hypersurfaces.
 These spaces are all conformally flat in the case of dimension four.

\medskip
Next to static spaces we consider a Miao-Tam critical metric \cite{MT2, MT}, which is a compact Riemannian manifold $(M, g)$ that admits a smooth nonzero solution $f$, vanishing at the smooth boundary of $M$, to
\begin{eqnarray} \label{0002bx3}
\nabla df = f(Rc -\frac{R}{n-1} g)- \frac{g}{n-1} .
\end{eqnarray}

In \cite{MT}, Miao-Tam critical metrics are classified when they are Einstein or conformally flat. In \cite{BDR},  Barros, Di\'{o}genes and Ribeiro proved that if $(M^4, g, f )$ is a Bach-flat simply connected, compact Miao-Tam critical metric with boundary isometric to a standard sphere $S^3$, then $(M^4, g)$ is isometric to a geodesic ball in a simply connected space form $\mathbb{R}^4, \mathbb{H}^4$ or $\mathbb{S}^4$.

\medskip
In \cite{CEM} Corvino, Eichmair and Miao defined a $V$-{\it static space} to be a Riemannian manifold $(M, g)$ which admits a non-trivial solution $(f, c)$, for a constant $c$,  to the equation
\begin{eqnarray} \label{0002bx3f}
\nabla df = f(Rc -\frac{R}{n-1} g)- \frac{c }{n-1}g.
\end{eqnarray}

Note that $(M, g)$ is a $V$-static space if and only if it admits a solution $f$ to (\ref{0002bx1}) or (\ref{0002bx3}) on $M$, seen by scaling constants.
Under a natural assumption, a  $V$-static metric $g$ is a critical point of a geometric functional as explained in the theorem 2.3 of \cite{CEM}.
Like static spaces, {\it local} $V$-static spaces are still important; see e.g. theorems 1.1, 1.6 and 2.3 in \cite{CEM}.

\medskip
Lastly, one may consider Riemannian metrics $(M,g)$ which admit a non-constant solution $f$ to
\begin{eqnarray} \label{0002bx2}
\nabla df = f(Rc -\frac{R}{n-1} g) +  Rc-  \frac{R}{n}g.
\end{eqnarray}
 If $M$ is a closed manifold, then $g$ is  a critical point of the total scalar curvature functional defined on the space of Riemannian metrics with unit volume and with constant scalar curvature on $M$.
By an abuse of terminology we shall call a metric $g$ satisfying (\ref{0002bx2}) {\it a critical point metric} even when $M$ is not closed.
 There are a number of literatures on this subject, including \cite[Section 4.F]{Be} and \cite{La, HCY, BR, QY}.

\medskip

In this paper we study spaces with harmonic curvature having a non-zero solution to (\ref{0002bx}). It is confined to
 four dimensional spaces here, but our study may be extendible to higher dimensions.
 As motivated by the Corvino's local deformation theory of scalar curvature,
 we study local  (i.e. not necessarily complete) classification.
 We completely characterize non-conformally-flat spaces, so that  together with the Kobayashi's work on conformally flat ones we get a full classification as follows.

\begin{thm} \label{local}
Let $(M, g)$ be a four dimensional  (not necessarily complete) Riemannian manifold with harmonic curvature, satisfying (\ref{0002bx}) with non-constant $f$.
Then for each point $p$ in some open dense subset $\tilde{M}$ of $M$, there exists a neighborhood $V$ of $p$ with one of the following properties;

\medskip

{\rm (i)}
 $(V,g)$ is isometric to a domain in   $ \mathbb{S}^2(\frac{R}{6}) \times   \mathbb{S}^2(\frac{R}{3})$ with $R>0$, where $\mathbb{S}^2(k) $ is the two-dimensional sphere with constant sectional curvature $k>0$.
And  $ f=c_1 \cos( \sqrt{\frac{R}{6}}s )  -x$  for any constant $c_1$, where $s$ is the distance on $\mathbb{S}^2(\frac{R}{6})$ from a point.
 The constant $R$ equals the scalar curvature of $g$. It holds that $x \frac{R}{3}   + y(R)= 0.$

\smallskip
{\rm (ii)} $(V,g)$ is isometric to a domain in $ (\mathbb{H}^2(\frac{R}{6}) \times   \mathbb{H}^2(\frac{R}{3}),    g_{\frac{R}{6}}  +  g_{\frac{R}{3}} )$ with $R<0$,   where $\mathbb{H}^2(k) $ is the hyperbolic plane with constant sectional curvature $k<0$ and   $g_{k}$ is the Riemannian metric of constant curvature $k$.
$g_{\frac{R}{6}}$ can be written as  $g_{\frac{R}{6}} = ds^2 + p(s)^2 dt^2 $ where $p^{''}+ \frac{R}{6}p =0$ and then $f= c_2 p^{'}-x$ for any constant $c_2$. It holds that $x \frac{R}{3}   + y(R)= 0.$

\smallskip

{\rm (iii)} $(V,g)$ is isometric to a domain in one of the static spaces in Example 3 of Subsection 2.1.2 below, which is the Riemannian product $(\mathbb{R}^1  \times W^3   , dt^2  +  ds^2 + h(s)^2 \tilde{g})$ of $(\mathbb{R}^1, dt^2)$  and some 3-dimensional conformally flat static space $(W^3, ds^2 + h(s)^2 \tilde{g})$ with zero scalar curvature. And $f = c \cdot h^{'}(s)- x$, for any constant $c$. It holds that $R=0$ and $y(0)=0$.

\smallskip

{\rm (iv)} $(V,g)$ is conformally flat. It is one of the metrics whose existence is described  in the section 2 of \cite{Ko}; $g= ds^2 + h(s)^2 g_k$ where
$h$ is a solution of
 \begin{eqnarray} \label{handf5}
 h^{''} +   \frac{R}{12}h = ah^{-3}, \ \ \  {\rm for} \   {\rm a}   \ {\rm constant} \  a.
 \end{eqnarray}
 For the constant $k$, the function $h$ satisfies
\begin{eqnarray} \label{handf6}
(h^{'})^2 + ah^{-2} + \frac{R}{12}h^2 =k.
 \end{eqnarray}

 And $f$ is a non-constant solution to the ordinary differential equation for $f$;
\begin{eqnarray} \label{handf4xxx}
h^{'}f^{'} - f h^{''}= x (h^{''} +\frac{R}{3}h)+ y(R)h.
 \end{eqnarray}

 Conversely, any $(V,g,f)$ from {\rm (i)}$\sim${\rm (iv)} has harmonic curvature and satisfies (\ref{0002bx}).

\end{thm}

Theorem \ref{local} only considered the case when $f$ is a nonconstant solution, but the other case of $f$ being a nonzero constant solution is easier, which is described in Subsection 2.1.1.

\medskip
Theorem \ref{local} yields a number of classification theorems on four dimensional spaces with harmonic curvature as follows.
Theorem \ref{complete}
classifies complete spaces satisfying (\ref{0002bx}).
Then Theorem \ref{locals},  \ref{localsxx} and \ref{localc} state the classification of {\it local} static spaces, $V$-static spaces and critical point metrics, respectively. Theorem \ref{completes} and \ref{completec}  classify complete static spaces and  critical point metrics, respectively.  Theorem \ref{compbdry}  gives a characterization of some 4-d Miao-Tam critical metrics with harmonic curvature,   which is comparable to the afore-mentioned Bach-flat result \cite{BDR}.

\medskip
To prove theorem \ref{local} we look into the eigenvalues  of the Ricci tensor, which is a Codazzi tensor under harmonic curvature condition.
This Codazzi tensor encodes some geometric information as investigated by Derdzinsky \cite{De}. In \cite{Ki} the first-named author has analyzed it in the Ricci soliton setting. We follow the same line of arguments.
A crucial part is to show that all Ricci-eigenvalues $\lambda_i$, $i=1,\cdots ,4$ locally depend on only one function $s$ such that
$\nabla  s= \frac{\nabla f}{ | \nabla f | }$.
Then we divide the proof into some cases, depending on the distinctiveness of  $\lambda_2, \lambda_3, \lambda_4$.
 There are two non-trivial cases; when these three are pairwise distinct and when exactly two of them are equal. In the latter case we reduce the analysis to ordinary differential equations, in a way similar to that of \cite{CC2}, and resolve them. And computations on (\ref{0002bx}) using Codazzi tensor property show that the former case does not occur.

\medskip
This paper is organized as follows. In section 2, we discuss examples and some properties from (\ref{0002bx}) and harmonic curvature.
In section 3, we prove that all Ricci-eigenvalues locally depend on only one variable.
We study in section 4 the case when the three eigenvalues $\lambda_2, \lambda_3, \lambda_4$ are pairwise distinct.
In section 5 and 6 we analyze the case when exactly two of the three are equal.
 In section 7 we prove the local classification theorem as Theorem 1.
We discuss the classification of complete spaces in section 8.
In section 9, 10 and 11 we treat static spaces, Miao-Tam critical \& $V$-static spaces and crtitcal point metrics respectively.

\section{Examples and properties from (\ref{0002bx}) and harmonic curvature}

We are going to describe  some examples of spaces which satisfy (\ref{0002bx}) in Subsection 2.1 and
state basic properties of spaces with harmonic curvature satisfying (\ref{0002bx}) in Subsection 2.2.

\subsection{Examples of spaces  satisfying (\ref{0002bx})}
\subsubsection{Spaces with a nonzero constant solution  to (\ref{0002bx})}

When $(M, g)$ has  a constant solution  $f= -x$ to (\ref{0002bx}), then  $y(R)+x\frac{R}{n-1}=0$.  Conversely, any metric with its scalar curvature satisfying $y(R) +x\frac{R}{n-1}=0$
admits the constant solution $f=-x$ to (\ref{0002bx}) because   $\nabla df = f(Rc -\frac{R}{n-1} g) + x Rc+ y(R) g = (f+x) (Rc -\frac{R}{n-1} g).$ This proves

\begin{lemma} \label{c001}
An $n$-dimensional Riemannian manifold $(M, g)$ of constant scalar curvature $R$ admits the constant solution $f=-x$ if and only if it satisfies $y(R)+x\frac{R}{n-1}=0$.
\end{lemma}

If $(M, g)$ has  a constant solution  $f= c_0$, which does not equal   $-x$, then $g$ is an Einstein metric.
Conversely, if $g$ is Einstein, i.e. $Rc = \frac{R}{n} g$ with $R \neq 0$, then any constant $c_0$ satisfying
$c_0 R = (n-1)xR + y(R)n(n-1)$ is a solution to  (\ref{0002bx}); but if $g$ is Ricci-flat, then $f= c_0$ is a solution exactly when
$y(0) =0$.

\medskip
\subsubsection{Some examples of spaces which satisfy (\ref{0002bx}) with non constant $f$}

\text{   }

\medskip
{\bf Example 1: Einstein spaces satisfying (\ref{0002bx}) with non constant $f$}

\medskip
Let $(M, g, f)$ be a 4-dimensional space satisfying (\ref{0002bx}) where $g$ is an Einstein metric. We shall show that $g$ has constant sectional curvature.
We may use the argument in the section 1 of Cheeger-Colding \cite{CC}. In fact,  the relation (1.6) of that paper corresponds to the equation
\begin{eqnarray} \label{cos07}
 \nabla df = \{- \frac{R }{12} f   + x \frac{R }{4}+ y(R) \}g
\end{eqnarray}
   in our Einstein case. One can readily see that their argument still works to get their (1.19);  in some neighborhood of any point in $M$ we can write $g=ds^2 + (f^{'}(s))^2 \tilde{g}$, where $s$ is a variable such that $\nabla s = \frac{ \nabla f }{|\nabla f|}$ and $ \tilde{g}$  is considered as a Riemannian metric on a level surface of $f$.

 As $g$ is Einstein, so is $ \tilde{g}$ from Derdzi\'{n}ski's Lemma 4  in \cite{De}.   As $ \tilde{g}$ is 3-dimensional,
 it has constant sectional curvature, say $k$.
And $f$ satisfies
$ f^{''} = - \frac{R }{12} f   + x \frac{R }{4}+ y(R)$, by feeding   $(\frac{\partial }{\partial s },\frac{\partial }{\partial s }  )$ to (\ref{cos07}).

Since $g$ is Einstein, we can readily see that our warped product metric $g$ has constant sectional curvature. In particular,
a 4-d complete positive Einstein space satisfying (\ref{0002bx}) with non constant $f$ is a round sphere; cf.  \cite{Ob, YN}.

\bigskip

{\bf Example 2}
Assume that $x \frac{R}{3}   + y(R)= 0.$  Then (\ref{0002bx}) reduces to   $\nabla d f = (f+x) (Rc -\frac{R}{n-1} g).$  This is the static space equation for $g$ and $F= f + x$.  We recall one example from \cite{La}.
On  the round sphere $ \mathbb{S}^2(1)$ of sectional curvature $1$,
we consider the local coordinates $(s, t) \in  (0,  \pi)  \times  \mathbb{S}^1$  so that the round metric is written $ds^2 +  {\sin^2 (s)} \ d t^2$. Let $f(s) = c_1 \cos s -x$  for any constant $c_1$.
Then the product metric of $ \mathbb{S}^2(1) \times  \mathbb{S}^2(2)$
with $f$ satisfies (\ref{0002bx}).

This example is not Einstein nor conformally flat.

\bigskip
{\bf Example  3}   Here we shall describe some 4-d non-conformally-flat static space $g_W + dt^2$.
We first recall some spaces  among
Kobayashi's warped product static spaces \cite{Ko} on  $I \times N(k)$ with  the metric $g = ds^2 + r(s)^2 \bar{g}$, where $I$ is an interval and $(\bar{g}, N(k))$ is a $(n-1)$-dimensional Riemannian manifold of constant sectional curvature $k$. And $f = c r^{'}$ for a nonzero constant $c$.

 In order for $g$ to be a static space, the next equation needs to be satisfied; for a constant $\alpha$
\begin{eqnarray} \label{m12}
r^{''}   + \frac{R}{n(n-1)}r=  \alpha r^{1-n},
\end{eqnarray}

\noindent along with an integrability condition; for a constant $k$,
\begin{eqnarray} \label{m13}
 (r^{'})^2 + \frac{2\alpha}{n-2}r^{2-n} + \frac{R}{n(n-1)}r^2 =k.
\end{eqnarray}
Existence of solutions depends on the values of $\alpha, R, k$.
Here we consider only when $R=0$.
Then there are three cases:

(i) $R=0, \ \alpha >0$,   $ \ \ \ \ \ \ $ (ii) $R=0, \ \alpha <0$,  $ \ \ \ \   \ \ \ \ $
(iii)  $R=0, \ \alpha =0$.

\medskip
The above (i), (ii) and (iii) correspond respectively to the case IV.1,  III.1 and II in Section 2 of \cite{Ko}.
The solutions for these cases are discussed in Proposition 2.5, Example 5 and Proposition 2.4 in that paper.
In particular, if $R=0, \alpha >0$ (then $k>0$) and $n=3$, we get the warped product metric on $\mathbb{R}^1 \times   \mathbb{S}^2(1) $ which contains the space section of the Schwarzshild space-time.
Next, if $R=0, \alpha <0$, then  there is an incomplete metric on $I \times N(k)$.
If $R=0, \alpha=0$,  then  $g$ is readily a flat metric.

 Let $(W^3, g_W, f)$ be one of the $3$-dimensional static spaces $(g, f)$ in the above paragraph.
We now consider the $4$-dimensional product metric $g_W + dt^2$ on $W^3   \times \mathbb{R}^1$. One can check that
$(W^3   \times \mathbb{R}^1, g_W + dt^2, f \circ {\rm pr}_1 )$ is a static space, where ${\rm pr}_1$ is the projection of $W^3   \times \mathbb{R}^1$ onto the first factor. The metric $g_W + dt^2$ is not conformally flat and  has three distinct $Ricci$ eigenvalues.

\subsection{Spaces with harmonic curvature}

We begin with a basic formula;

 \begin{lemma} \label{threesolx}
For a $4$-dimensional manifold $(M^4, g, f)$ with harmonic curvature satisfying {\rm (\ref{0002bx})}, it holds that
\begin{eqnarray*} \label{solbax}
-R(X,Y,Z, \nabla f) = - R(X,Z) g(\nabla f, Y)  + R(Y,Z)g(\nabla f, X)  \\
 - \frac{R}{3} \{ g(\nabla f, X)  g(Y,Z) - g(\nabla f, Y) g(X,Z) \}   .
\end{eqnarray*}
\end{lemma}

\begin{proof}
By Ricci identity, $\nabla_i  \nabla_j  \nabla_k f   - \nabla_j  \nabla_i  \nabla_k f   = -R_{ijkl} \nabla_l f$.
The equation {\rm (\ref{0002bx})} gives
\begin{eqnarray*}
-R_{ijkl} \nabla_l f &=\nabla_i \{ f (R_{jk}  - \frac{1}{3} R g_{jk}) + x R_{jk}+ y(R) g_{jk} \}   \hspace{2cm} \\  & - \nabla_j \{ f (R_{ik}  - \frac{1}{3} R g_{ik}) + x R_{ik}+ y(R) g_{ik} \}    \\
&  =f_i (R_{jk}  - \frac{1}{3} R g_{jk})        - f_j (R_{ik}  - \frac{1}{3} R g_{ik}), \hspace{2cm}
\end{eqnarray*}
 which yields the lemma.
\end{proof}

A Riemannian manifold with harmonic curvature is real analytic in harmonic coordinates \cite{DG}.
The equation (\ref{0002bx}) then tells that $f$ is real  analytic in harmonic coordinates.

\medskip
One may mimic arguments in \cite{CC2} and get the next lemma.

 \begin{lemma} \label{threesolbx}
Let  $(M^n, g, f)$ have harmonic curvature, satisfying {\rm (\ref{0002bx})}  with nonconstant $f$. Let $c$ be a regular value of $f$ and $\Sigma_c= \{ x | f(x) =c  \}$  be the level surface of $f$. Then the followings hold;

{\rm (i)} Where $\nabla f \neq 0$,  $E_1 := \frac{\nabla f }{|\nabla f  | }$ is an eigenvector field of $Rc$.

{\rm (ii)} $ |\nabla f|$  is constant on a connected component of $\Sigma_c$.

{\rm (iii)} There is a function $s$ locally defined with   $s(x) = \int  \frac{   d f}{|\nabla f|} $, so that

$ \ \ \ \ ds =\frac{   d f}{|\nabla f|}$ and $E_1 = \nabla s$.

{\rm (iv)}  $R({E_1, E_1})$ is constant on a connected component of $\Sigma_c$.

{\rm (v)}  Near a point in $\Sigma_c$, the metric $g$ can be written as

$\ \ \ g= ds^2 +  \sum_{i,j > 1} g_{ij}(s, x_2, \cdots  x_n) dx_i \otimes dx_j$, where
    $x_2, \cdots  x_n$ is a local

 $ \ \ \ $   coordinates system on $\Sigma_c$.

{\rm (vi)}  $\nabla_{E_1} E_1=0$.
\end{lemma}

\begin{proof} In Lemma \ref{threesolx},  put $Y=Z= \nabla f$ and $X \perp \nabla f$ to get

$0=-R(X,\nabla f,\nabla f, \nabla f) = - R(X,\nabla f) g(\nabla f, \nabla f)    $. So, $R(X,\nabla f)=0$. Hence $E_1 = \frac{\nabla f }{|\nabla f  | }$ is an eigenvector of $Rc$.
Also, $\frac{1}{2}\nabla_X |\nabla f|^2 =\langle \nabla_X \nabla f, \nabla f \rangle =   f R( \nabla  f, X )  = 0$  for $X \perp \nabla f$. We proved {\rm (ii)}.
Next $d (\frac{   d f}{|\nabla f|}) = -\frac{1}{2 |\nabla f|^{\frac{3}{2}}} d |\nabla f|^{2} \wedge df= 0  $ as $\nabla_X ( |\nabla f|^2 )=0$. So, {\rm (iii)} is proved.
 It holds that $(\nabla_{E_1}E_1) f=0$ because, setting $\nabla_{E_1}E_1( f)=\sum_i \langle \nabla_{E_1}E_1, E_i \rangle E_i (f) = \langle \nabla_{E_1}E_1, E_1 \rangle E_1 (f)$, then
$ 2\langle \nabla_{E_1}E_1, E_1 \rangle = E_1 \langle E_1, E_1 \rangle =0$.
We  compute
$\nabla_g df (E_1, E_1)= f (Ric_g  - \frac{1}{n-1} R_g g)(E_1, E_1)+xR(E_1, E_1) + y(R)$.  Then $ E_1 E_1 f - (\nabla_{E_1}E_1) f = E_1 E_1 f= f (R({E_1, E_1}) - \frac{1}{n-1} R)+xR(E_1, E_1) + y(R)$. Since $f$ is not zero on an open subset,  so $R({E_1, E_1})$ is constant on a connected component of $\Sigma_c$.
 As $\nabla f$ and the level surfaces of $f$ are perpendicular, one gets  {\rm (v)}.

One uses {\rm (v)} to compute Christoffel symbols and gets  {\rm (vi)}.
\end{proof}

The Ricci tensor $r$ of a Riemannian metric with harmonic curvature is a Codazzi tensor, written in local coordinates as $\nabla_k R_{ij}  = \nabla_i R_{kj}$.
Here $Rc$ or $r$ denotes the Ricci tensor, but its components in vector frames shall be written as $R_{ij}$.
Following Derdzi\'{n}ski  \cite{De}, for a point $x$ in $M$, let  $E_r(x)$ be the number of distinct eigenvalues of $r_x$,
and set $M_r = \{    x \in M \  |  \  E_r {\rm \ is \ constant \ in \ a \ neighborhood of \ } x \}$, so that $M_r$ is an open dense subset of $M$.
 Then we have;

\begin{lemma} \label{abc60x} For a Riemannian metric $g$ of dimension $n \geq 4$ with harmonic curvature, consider orthonormal vector fields $E_i$, $i=1, \cdots n$ such that
$R(E_i, \cdot ) = \lambda_i g(E_i, \cdot)$. Then the followings hold
in each connected component of $M_r$;

\bigskip
\noindent {\rm (i)}
 $(\lambda_j - \lambda_k ) \langle \nabla_{E_i} E_j, E_k \rangle   + {E_i} \{R(E_j, E_k)   \}=(\lambda_i - \lambda_k ) \langle \nabla_{E_j} E_i, E_k\rangle + {E_j} \{R(E_k, E_i)   \}, \ \ $

 for any $i,j,k =1, \cdots n$.

\smallskip
\noindent {\rm (ii)}  If $k \neq i$ and $k \neq j$,
$ \ \ (\lambda_j - \lambda_k ) \langle \nabla_{E_i} E_j, E_k\rangle=(\lambda_i - \lambda_k ) \langle \nabla_{E_j} E_i, E_k\rangle .$

\smallskip
\noindent {\rm (iii)} Given distinct eigenfunctions $\lambda, \mu$ of $A$ and local vector fields $v, u$ such that  $A v = \lambda v$, $Au = \mu u$ with $|u|=1$, it holds that

$ \ \ \ \ \  v(\mu) = (\mu - \lambda) <\nabla_u u, v > $.

\smallskip
\noindent {\rm (iv)} For each eigenfunction $\lambda$, the $\lambda$-eigenspace distribution is integrable and its leaves are totally umbilic submanifolds of $M$.

\end{lemma}

\begin{proof}
 The statement {\rm (i)} was proved in \cite{Ki}.  And {\rm (ii)} and {\rm (iii)} follow from {\rm (i)}.
{\rm (iii)} and {\rm (iv)}  are from the section 2 of \cite{De}.
\end{proof}

\bigskip
  Given $(M^n,g,f)$ with harmonic curvature satisfying (\ref{0002bx}), $f$ is real  analytic in harmonic coordinates, so $\{ \nabla f \neq 0  \}$ is open and dense in $M$.
 Lemma \ref{threesolbx} gives that
for any point $p$ in the open dense subset $M_{r} \cap \{ \nabla f \neq 0  \}$ of $M^n$,
there is a neighborhood $U$ of $p$ where there exists an orthonormal Ricci-eigen vector fields $E_i$, $i=1, \cdots  , n$  such  that

{\rm (i)}  $E_1= \frac{\nabla f}{|\nabla f| }$,

 {\rm (ii)} for $i>1$, $E_i$ is tangent to smooth level hypersurfaces of $f$.

\medskip
These local orthonormal Ricci-eigen vector fields $\{ E_i \}$  shall be called an {\it adapted frame field} of $(M, g, f)$.

\section{Constancy of $\lambda_i$ on level hypersurfaces of $f$}

For an adapted frame field of $(M^n,g,f)$ with harmonic curvature satisfying (\ref{0002bx}), we set $ \zeta_i:= - \langle   \nabla_{E_i}  E_i ,  E_1  \rangle=  \langle    E_i , \nabla_{E_i}  E_1  \rangle$, for $i >1$. Then
$\nabla_{E_i}  E_1 = \nabla_{E_i} (\frac{\nabla f}{  | \nabla f |}) =   \frac{ \nabla_{E_i} \nabla f }{  | \nabla f |}=   \frac{ f R({E_i}, \cdot) -  f\frac{R}{n-1} g( {E_i}, \cdot  ) +xR({E_i}, \cdot ) + y(R)g({E_i}, \cdot ) }{  | \nabla f |} $. So we may write;
\begin{equation} \label{lambda06ax}
\nabla_{E_i}  E_1 =   \zeta_i E_i     \ \    {\rm where}   \   \zeta_i =     \frac{(f+x) R(E_{i}, E_i)  - \frac{R}{n-1}f + y(R) }{| \nabla f|}.
\end{equation}

Due to Lemma \ref{threesolbx}, in a neighborhood of a point $p \in M_{r} \cap \{ \nabla f \neq 0  \}$, $f$  may be considered as functions of the variable $s$ only, and we write the derivative in $s$ by a prime: $f^{'} = \frac{df}{ds}$.

\begin{lemma} \label{abc60byx} Let $(M, g, f)$ be a 4-dimensional space  with harmonic curvature, satisfying (\ref{0002bx}) with nonconstant $f$.
The Ricci eigen-functions $\lambda_i$ associated to an adapted frame field $E_i$
are constant on a connected component of a regular level hypersurface $\Sigma_c$ of $f$, and so depend on the local variable  $s$ only. And $\zeta_i$, $i=2,3,4$, in  $\rm{(\ref{lambda06ax})}$ also depend on $s$ only.
 In particular, we have
$E_i (\lambda_j) = E_i (\zeta_k)= 0$ for $i,k >1$ and any $j$.
\end{lemma}

\begin{proof}
We use  $f_{ij} = f (R_{ij}  - \frac{1}{3} R g_{ij}) + x R_{ij}+ y(R) g_{ij}$ to compute;

$ \sum_{j=1}^4 \frac{1}{2} \nabla_{E_j} \nabla_{E_j}  ( |\nabla f|^2  ) = \sum_{i,j} \frac{1}{2} \nabla_{E_j} \nabla_{E_j}  (f_i f_i  )$

$\hspace{0.9cm} =\sum_{i,j}   \nabla_{E_j}  (f_i f_{ij}  )  =\sum_{i,j}  \nabla_{E_j}  \{ f f_i (R_{ij} -\frac{R}{3}g_{ij}) + xf_iR_{ij}+ y(R)f_ig_{ij} \} $

$ \hspace{0.9cm} = \sum_{i,j} f_j f_i (R_{ij} -\frac{R}{3}g_{ij}) +  f f_{ij} (R_{ij} -\frac{R}{3}g_{ij})+  xf_{ij}R_{ij} + y(R)f_{ij}g_{ij}$

$\hspace{0.9cm} = (R_{11} -  \frac{R}{3}) |\nabla f|^2  +     \sum_{i,j}  (f+x)^2R_{ij} R_{ij}-  \frac{2R^2}{9}f^2   -\frac{2xR^2}{3} f $

$  \hspace{1.4cm}+(2x- \frac{2f}{3})y(R) R       + 4y(R)^2$.

\noindent
Meanwhile,

$ \sum_{j=1}^4\nabla_{E_j} \nabla_{E_j}  ( |\nabla f|^2  ) = \sum_{j=1}^4 {E_j} {E_j}( |\nabla f|^2  )  -  (\nabla_{E_j} E_j)  ( |\nabla f|^2  )$

$ \hspace{3.7cm} =( |\nabla f|^2  )^{''} +  \sum_{j=2}^4 \zeta_j  ( |\nabla f|^2  )^{'} $.

Since
 $R$ and $\lambda_1= R_{11}$ depend on $s$ only  by Lemma  \ref{threesolbx}, the function $\sum_{j=2}^4 \zeta_j $  depends only on $s$ by (\ref{lambda06ax}).
We compare the above two expressions of $\sum_{j=1}^4\nabla_{E_j} \nabla_{E_j}  ( |\nabla f|^2  )$ to see that
 $\sum_{i,j}  (f+x)^2R_{ij} R_{ij}$ depends only on $s$.
As $f$ is nonconstant real analytic,  $\sum_{i,j} R_{ij}R_{ij}$  depends only on $s$.

\medskip
Below we drop summation symbols.
\begin{eqnarray*}
 \nabla_k (f_i f_{ij} R_{jk})=   \nabla_k \{f_iR_{jk} (f (R_{ij}  - \frac{1}{3} R g_{ij}) + x R_{ij}+ y(R) g_{ij}) \} \hspace{2cm}\nonumber \\
 =   \nabla_k \{f_i ((f+x) R_{ij}R_{jk}  - \frac{f}{3} R R_{ik} + y(R) R_{ik}) \} \hspace{4.3cm}\nonumber \\
=   f_{ik} \{(f+x) R_{ij}R_{jk}  - \frac{f}{3} R R_{ik} + y(R) R_{ik}\}  \hspace{5cm} \nonumber \\
  +  f_i (f_k R_{ij}R_{jk}+(f+x) R_{jk}\nabla_k R_{ij}  - \frac{f_k}{3} R R_{ik})  \hspace{4cm} \nonumber \\
  =   \{ (f+x) R_{ik}  - \frac{f}{3} R g_{ik} + y(R) g_{ik}   \} \{(f+x) R_{ij}R_{jk}  - \frac{f}{3} R R_{ik} + y(R) R_{ik}\}  \nonumber \\
  +   f_if_k R_{ij}R_{jk}+(f+x) f_iR_{jk}\nabla_k R_{ij}  - \frac{f_if_k}{3} R R_{ik} \hspace{3.5cm}  \nonumber \\
    =  (f+x)^2 R_{ik} R_{ij}R_{jk} +(f+x) f_iR_{jk}\nabla_k R_{ij}    + L(s),  \hspace{3.7cm}  \nonumber
\end{eqnarray*}
where $L(s)$ is a function of $s$ only.

\medskip
Using $\nabla_k R_{ij}=\nabla_i R_{jk}$, we get
\begin{eqnarray} \label{rrt3}
\ \ \ \ \ \ \ \ \ \ \  \nabla_k (f_i f_{ij} R_{jk})  =  (f+x)^2 R_{ik} R_{ij}R_{jk} +\frac{(f+x)}{2} f_i\nabla_i (R_{jk} R_{jk})    + L(s).
\end{eqnarray}

All terms except $(f+x)^2 R_{ij}R_{jk}R_{ik} $ in the right hand side of (\ref{rrt3}) depend on $s$ only.
From the constancy of $R$ and (\ref{lambda06ax}) we also get
\begin{eqnarray} \label{rrtb3}
2\nabla_k (f_i f_{ij} R_{jk})& =\nabla_k (2 f_i f_{ij}) \cdot R_{jk}
 =\nabla_k  \nabla_j (f_i f_{i}) \cdot R_{jk}
 \hspace{3.9cm} \nonumber \\
  &= \sum_{j, k,i } E_k  E_j (f_i f_{i}) \cdot R_{jk}- (\nabla_{E_k} E_j )(f_i f_{i}) \cdot R_{jk}
 \hspace{2.3cm} \nonumber \\
   &= \sum_{j, i } E_j  E_j (f_i f_{i}) \cdot R_{jj}- (\nabla_{E_j} E_j )(f_i f_{i}) \cdot R_{jj}
 \hspace{2.5cm} \nonumber \\
   & = \sum_{ i } E_1  E_1 (f_i f_{i}) \cdot R_{11}+ \sum_{j=2}^4 \zeta_j E_1(f_i f_{i}) \cdot  R_{jj}
 \hspace{2.3cm} \nonumber\\
   & =  (|\nabla f |^2)^{''} \cdot R_{11}+ \sum_{j=2}^4 \frac{(f+x)R_{jj}  R_{jj} - \frac{R}{3}f R_{jj} + y(R)R_{jj}}{| \nabla f|}   E_1(|\nabla f |^2)
\end{eqnarray}
which depends only on $s$.

So, we compare  (\ref{rrt3}) with (\ref{rrtb3}) to see that
 $R_{ij}R_{jk}R_{ik}$ depends only on $s$.
  Now $\lambda_1$ and  $\sum_{i=1}^4 (\lambda_i)^k$, $k=1, \cdots ,3$,  depend only on $s$.
This implies that
each $\lambda_i$, $i=1, \cdots ,4$, depends only on $s$.
By (\ref{lambda06ax}), $\zeta_i$, $i=2, 3 ,4$ depend on $s$ only.

\end{proof}

\section{4-dimensional space with distinct  $\lambda_2, \lambda_3, \lambda_4$}

Let $(M,g,f)$ be a four dimensional Riemannian manifold with harmonic curvature satisfying (\ref{0002bx}).
For an adapted frame field $\{ E_j \} $ with its eigenfunction $\lambda_j$ in an open subset of $M_{r} \cap \{ \nabla f \neq 0  \}$,
we may only consider three cases depending on the distinctiveness of $\lambda_2,\lambda_3,\lambda_4$;
the first case is when  $ \lambda_i$, $i=2,3,4$ are all equal (on an open subset),
 and  the second  is when exactly two of the three are equal. And the last is when the three  $ \lambda_i$, $i=2,3,4$, are mutually different.
In this section we shall study the last case. We set $\Gamma^k_{ij}:= < \nabla_{E_i} E_j,  E_k>$.

\begin{lemma}\label{77bx}
Let $(M,g,f)$ be a four dimensional Riemannian manifold with harmonic curvature satisfying (\ref{0002bx}) with nonconstant $f$.
Suppose that for an adapted frame fields $E_j$, $j=1,2,3,4$,
in an open subset $W$ of $M_{r} \cap \{ \nabla f \neq 0  \}$,   the eigenfunctions $\lambda_2, \lambda_3, \lambda_4$ are distinct from each other. Then the following holds in $W$;

\medskip
\noindent For distinct $\ i, j>1$,
$ \ \ R_{1ii1} =-\zeta_i^{'}  -  \zeta_i^2 , $ $\ \ \ \    R_{1ij1}= 0$.

$R_{11} =  -\zeta_2^{'}
 -  \zeta_2^2   -\zeta_3^{'}
 -  \zeta_3^2  -\zeta_4^{'}
 -  \zeta_4^2 .$

$R_{22} = - \zeta_2^{'}
 -  \zeta_2^2 -\zeta_2 \zeta_3
 -\zeta_2 \zeta_4  -2 \Gamma_{34}^2 \Gamma_{43}^2 .$

$R_{33} = -\zeta_3^{'}
 -  \zeta_3^2 -\zeta_3 \zeta_2
 -\zeta_3 \zeta_4  +2 \frac{( \zeta_2 - \zeta_4 )}{ \zeta_3 - \zeta_4   }\Gamma_{34}^2 \Gamma_{43}^2 .$

$R_{44} = -\zeta_4^{'}
 -  \zeta_4^2 -\zeta_4 \zeta_2
 -\zeta_4 \zeta_3  +2\frac{( \zeta_2 - \zeta_3 )}{ \zeta_4 - \zeta_3   } \Gamma_{34}^2 \Gamma_{43}^2 .$

$  R_{1j1j} =  R_{jj} - \frac{R}{3}  $.

\end{lemma}

\begin{proof} $\nabla_{E_1}  E_1 =0$ from Lemma \ref{threesolbx} {\rm (vi)} and
 $\nabla_{E_i} E_1 = \zeta_i E_i$ from (\ref{lambda06ax}).  Let  $i, j >1$ be distinct. From Lemma  \ref{abc60x} (iii) and Lemma \ref{abc60byx},
 $ \langle  \nabla_{E_i} E_i ,   E_j\rangle=0$.
  And  $ \langle  \nabla_{E_i} E_i ,   E_1\rangle= - \langle   E_i ,   \nabla_{E_i} E_1\rangle= - \zeta_i$. So, we get  $\nabla_{E_i}  E_i = -\zeta_i E_1  $.
Now, $ \langle  \nabla_{E_i} E_j, E_i \rangle=-\langle  \nabla_{E_i} E_i, E_j \rangle=0$,
 $\langle  \nabla_{E_i} E_j, E_j \rangle=0  $. And $\langle\nabla_{E_i} E_j, E_1 \rangle = -\langle   \nabla_{E_i} E_1 ,  E_j\rangle =0 $. So,  $\nabla_{E_i} E_j= \sum_{k \neq 1,i,j}\Gamma_{ij}^k E_k$.  Clearly  $\Gamma_{ij}^k =- \Gamma_{ik}^j $.
From Lemma  \ref{abc60x} {\rm (ii)}, $(\lambda_i - \lambda_j ) \langle \nabla_{E_1} E_i, E_j\rangle=(\lambda_1 - \lambda_j ) \langle \nabla_{E_i} E_1, E_j\rangle $.  As $\langle \nabla_{E_i} E_1, E_j\rangle=0 $,  $\langle\nabla_{E_1} E_i, E_j\rangle=0$.     This gives $\nabla_{E_1} E_i =0$. Summarizing, we have got;

\medskip
\noindent For $i, j >1$, $i \neq j$,

$\nabla_{E_1}  E_1 =0$, $\ \ \ \ \ \nabla_{E_i} E_1 = \zeta_i E_i $,
$\ \ \ \ \ \  \nabla_{E_i}  E_i = -\zeta_i E_1 $, $ \ \ \ \nabla_{E_1} E_i=0$.

$\nabla_{E_i} E_j= \sum_{k \neq 1,i,j} \Gamma_{ij}^k E_k $.

\medskip
One uses  Lemma \ref{abc60byx} in computing curvature components. We get $ R_{1ii1} =-\zeta_i^{'}  -  \zeta_i^2$, and
for distinct $\ i, j, k>1$,
$ \ \ R_{jiij} = -\zeta_j \zeta_i  -  \Gamma_{ji}^k \Gamma_{ik}^j  - \Gamma_{ji}^k \Gamma_{ki}^j + \Gamma_{ij}^k \Gamma_{ki}^j$ and $ R_{kijk} = E_k (\Gamma^k_{ij})$. And $R_{1ij1}= 0$.

\medskip

From Lemma \ref{abc60x}, for distinct $i,j,k >1$,  we have
\begin{equation} \label{ree2x}
(\zeta_j - \zeta_k ) \Gamma^k_{ij}=(\zeta_i - \zeta_k ) \Gamma^k_{ji},
  \end{equation}
\noindent which helps to express $R_{ii}$.
 Lemma \ref{threesolx} gives $-R(E_1,E_j,E_j, \nabla f) =  (R_{jj} - \frac{R}{3} )g(\nabla f, E_1)$ for $j >1$. From this we get
\begin{equation} \label{reex}
 R_{1j1j} =  R_{jj} - \frac{R}{3}.
\end{equation}

\end{proof}

\noindent
From the proof of the above Lemma, we may write
\begin{equation} \label{coeff01}
[E_2, E_3] = \alpha E_4, \ \ \  [E_3, E_4] = \beta E_2,  \ \ \ [E_4, E_2] = \gamma E_3.
\end{equation}

From Jacobi identity $[[E_1, E_2], E_3]   +   [[E_2, E_3], E_1] + [[E_3, E_1], E_2]=0  $,  we have
\begin{equation} \label{4dformb1x}
E_1(\alpha) = \alpha ( \zeta_4 -  \zeta_2 - \zeta_3    ).
 \end{equation}

And (\ref{ree2x}) gives;
\begin{equation} \label{4dformb2x}
\beta =   \frac{(\zeta_3 - \zeta_4   )^2}{(\zeta_2 - \zeta_3   )^2} \alpha,     \ \ \ \ \  \gamma =   \frac{(\zeta_2 - \zeta_4   )^2}{(\zeta_2 - \zeta_3   )^2} \alpha.
\end{equation}

\bigskip

\noindent
We  set $a:= \zeta_2$, $b:= \zeta_3$ and  $c:= \zeta_4$. From (\ref{0002bx}), $\zeta_i f^{'} = f (R_{ii}- \frac{R}{3})+xR_{ii} +y(R) $ for $i >1$. With this and Lemma \ref{77bx},

\bigskip

$ (a - b)\frac{f^{'}}{f} =(1 + \frac{x}{f})(R_{22} - R_{33}) =  (1 + \frac{x}{f}) \{(b  -a )c
 -2\{ 1+  \frac{( a - c )}{ b - c   }\}\Gamma_{34}^2 \Gamma_{43}^2  \} $.
So,
\begin{equation} \label{fprix}
 -\frac{f^{'}}{f} = (1 + \frac{x}{f}) \{c  +2 \frac{( a+ b   - 2c )}{ (a - b)(b - c )  }\Gamma_{34}^2 \Gamma_{43}^2\}.
 \end{equation}
Similarly,
 $(a - c)\frac{f^{'}}{f} =(1 + \frac{x}{f}) (R_{22} - R_{44}) = (1 + \frac{x}{f}) \{ (c  -a )b
 -2\{ 1+  \frac{( a - b )}{ c - b   }\}\Gamma_{34}^2 \Gamma_{43}^2 \}  $.
 So,
\begin{equation} \label{fprx}
-\frac{f^{'}}{f}= (1 + \frac{x}{f}) \{ b
 +2  \frac{( a + c - 2b )}{(a - c) (c - b )  }\Gamma_{34}^2 \Gamma_{43}^2 \}.
 \end{equation}
From (\ref{fprix}) and (\ref{fprx}),
we get
\begin{equation} \label{fpr31x}
 4\Gamma_{34}^2 \Gamma_{43}^2 = \frac{(a - b)(a - c) (b - c )^2}{ ( a^2  +b^2+ c^2 - ab  -  bc - ac ) },
  \end{equation}

\begin{equation} \label{fpr2x}
-\frac{f^{'}}{f}=  (1 + \frac{x}{f}) \frac{a^2  b +   a^2 c  + a   b^2 +  a  c^2 +  b^2 c +   c^2   b   - 6a  b   c }{ 2 ( a^2  +b^2+ c^2 - ab  -  bc - ac ) }.
 \end{equation}

\bigskip
The formula (\ref{reex}) gives $  R_{1212} -  R_{1313}=  R_{22} - R_{33}$, which reduces to
\begin{eqnarray} \label{15eqnxb}
2(a^{'}   - b^{'})  &= -2(a^2 - b^2)+
 b c-a c  +\frac{(a - b)(b - c) (c - a )( a +b  - 2c )}{ 2( a^2  +b^2+ c^2 - ab  -  bc - ac ) }  \nonumber \\
 &=  -  2(a^2 - b^2)  +  \frac{(a-b)}{2P}  A,   \hspace{4cm}
 \end{eqnarray}
where we set $P:= a^2  +b^2+ c^2 - ab  -  bc - ac $ and $A:= 6abc -a^2b -ab^2 -a^2c -ac^2 -b^2c -bc^2   $. By symmetry, we get
\begin{eqnarray} \label{15eqnx}
\zeta_i^{'}   - \zeta_j^{'} =  -  (\zeta_i^2 - \zeta_j^2)  +  \frac{(\zeta_i-\zeta_j)}{4P}  A,   \ \ {\rm for}  \ i, j \in \{2,3, 4\}.
 \end{eqnarray}

\bigskip
The formula (\ref{15eqnx}) looks  different from the corresponding one in the soliton's study  in \cite{Ki}; $ \ \zeta_i^{'}   - \zeta_j^{'} =  -  (\zeta_i^2 - \zeta_j^2)$. But surprisingly the next proposition still works through in resolving  (\ref{0002bx}); refer to the proposition 3.4 in \cite{Ki}. Here the formula (\ref{fpr2x}) is crucial.

\begin{prop} \label{4dformcx}
Let $(M,g,f)$ be a four dimensional Riemannian manifold with harmonic curvature, satisfying (\ref{0002bx}) with nonconstant $f$.
For any adapted frame field $E_j$, $j=1,2,3,4$,
in an open dense subset $M_{r} \cap \{ \nabla f \neq 0  \}$ of $M$,
 the three eigenfunctions $\lambda_2, \lambda_3, \lambda_4$ cannot be  pairwise distinct, i.e. at least two of the three coincide.

\end{prop}
\begin{proof}
Suppose that  $\lambda_2, \lambda_3, \lambda_4$ are pairwise distinct. We shall prove then that $f$ should be a constant, a contradiction to the hypothesis.

 \noindent
 \medskip
From (\ref{fpr31x}) and (\ref{ree2x}),

$ (\alpha - \gamma + \beta)^2= 4(\Gamma_{34}^2)^2  = 4 \Gamma_{34}^2 \Gamma_{43}^2 \frac{(a-b)}{(a-c)}  =\frac{(a - b)^2 (b - c )^2}{ ( a^2  +b^2+ c^2 - ab  -  bc - ac ) }$.

From (\ref{4dformb2x}),
\begin{eqnarray*}
(\alpha - \gamma + \beta)^2 =\alpha^2 \{1 -  \frac{(a - c   )^2}{(a - b   )^2}  +  \frac{(b - c   )^2}{(a - b   )^2}  \}^2  = \frac{4 \alpha^2 (b  - c )^2}{(a - b   )^2}.
 \end{eqnarray*}

So,   $ \alpha^2    = \frac{(a - b)^4 }{ 4P } .$ Since $a,b,c$ are all functions of $s$ only, so is $\alpha$.
  We compute from (\ref{15eqnx})

  \begin{eqnarray} \label{13eqnx}
  (a-b) (a^{'} -b^{'})  + (a-c) (a^{'} -c^{'})+ (b-c) (b^{'} -c^{'})  \hspace{4cm} \nonumber \\
  =  -(a - b)  (a^2 - b^2)  -(a - c)  (a^2 - c^2)  -(b - c)  (b^2 - c^2)   \hspace{1.3cm} \nonumber \\
   +  \frac{A}{4P}\{(a-b)^2 + (a-c)^2 +(b-c)^2  \}              \hspace{4.5cm} \nonumber \\
 = -2(a^3 +b^3 + c^3) +a^2b +ab^2 + a^2c  +ac^2+ b^2 c + bc^2  +  \frac{A}{2}              \hspace{0.7cm} \nonumber \\
    =-2(a^3 +b^3 + c^3 -3abc) - \frac{1}{2}A    \hspace{5.2cm}
 \end{eqnarray}
Differentiating $ \alpha^2    = \frac{(a - b)^4 }{ 4P } $ in $s$ and using (\ref{15eqnx}) and (\ref{13eqnx}),

  \begin{eqnarray*}
 2 \alpha \alpha^{'}     =  &  \frac{(a - b)^3 (a^{'} - b^{'}) }{ P}  -   \frac{(a - b)^4 ( 2a a^{'}  +2b b^{'} + 2c c^{'}  - a b^{'} - b a^{'} -  a c^{'} -c a^{'} - c b^{'}- b c^{'}  )  }{ 4P^2 } \hspace{1.5cm} \\
 =  &  \frac{-(a - b)^3 (a^2 - b^2) }{ P} +  \frac{(a-b)^4}{4P^2}  A  -   \frac{(a - b)^4 \{ (a-b) (a^{'} -b^{'})  + (a-c) (a^{'} -c^{'})+ (b-c) (b^{'} -c^{'})    \}  }{ 4P^2 }  \\
 =  & - \frac{(a - b)^4 (a+b) }{ P} +  \frac{(a-b)^4}{4P^2}  A   +  \frac{(a - b)^4 \{ 2(a^3 +b^3 + c^3 -3abc)   \}  }{ 4P^2 }  +  \frac{(a - b)^4 \{ \frac{1}{2}A  \}  }{ 4P^2 }  \hspace{1cm} \\
= &-\frac{(a - b)^4  }{ P} \frac{(a+b-c)}{2}  +  \frac{3(a-b)^4}{8P^2}  A \hspace{7cm}
 \end{eqnarray*}

Meanwhile, from (\ref{4dformb1x}) and  $ \alpha^2    = \frac{(a - b)^4 }{ 4P } $,

\begin{equation*}
 2 \alpha \alpha^{'}= 2 \alpha^2 ( c -  a - b    )= -\frac{(a - b)^4 }{ 2P }( a +  b - c    ).
 \end{equation*}

\noindent Equating these two expressions for $ 2 \alpha \alpha^{'}$, we get $A=0$.
From (\ref{fpr2x}), $f^{'} =0$.
\end{proof}

\section{4-dimensional space with $\lambda_2 \neq \lambda_3 =\lambda_4$ }

In this section we study when exactly two of $\lambda_2, \lambda_3, \lambda_4$ are equal.  We may well assume that $ \lambda_2  \neq \lambda_3=  \lambda_4 $. We use (\ref{lambda06ax}), Lemma \ref{abc60x} and Lemma \ref{abc60byx} to compute $\nabla_{E_i} E_j$'s and get the next lemma.

\begin{lemma} \label{claim112nax}
Let $(M,g,f)$ be a four dimensional Riemannian manifold with harmonic curvature satisfying (\ref{0002bx}) with nonconstant $f$.
Suppose that  $ \lambda_2  \neq \lambda_3=  \lambda_4$ for an adapted frame fields $E_j$, $j=1,2,3,4$,
on an open subset $U$ of $M_{r} \cap \{ \nabla f \neq 0  \}$.
Then we have;

 \medskip
 $[E_1, E_2]= -\zeta_2 E_2  $ and $[E_3, E_4]= \beta_3 E_3 + \beta_4 E_4,$  for some functions $\beta_3$, $\beta_4$.

\smallskip
In particular, the distribution spanned by $E_1$ and $E_2$ is integrable. So is that  spanned by $E_3$ and $E_4$.
\end{lemma}

\begin{proof}
From Lemma \ref{abc60x} {\rm (ii)} and (\ref{lambda06ax}),
$(\lambda_2 - \lambda_i ) \langle \nabla_{E_1} E_2, E_i\rangle=(\lambda_1 - \lambda_i ) \langle \nabla_{E_2} E_1, E_i\rangle=(\lambda_1 - \lambda_i) \langle\zeta_2 E_2, E_i\rangle=0 ,$ for $i=3,4$. This gives $\nabla_{E_1}  E_2=  0$, and so $[E_1, E_2]= -\zeta_2 E_2  $.

From Lemma \ref{abc60x} {\rm (ii)},
$(\lambda_2 - \lambda_4 ) \langle \nabla_{E_3} E_2, E_4\rangle=(\lambda_3 - \lambda_4 ) \langle \nabla_{E_2} E_3, E_4\rangle=0 .$
So,  $\langle\nabla_{E_3}  E_2, E_4\rangle = - \langle  E_2,\nabla_{E_3} E_4\rangle =0$. This and  (\ref{lambda06ax}) yields
$\nabla_{E_3}  E_4 = \beta_3 E_3  $, for some function $ \beta_3$. Similarly,  $\nabla_{E_4}  E_3 = - \beta_4 E_4  $ for some function $ \beta_4$.
Then  $[E_3, E_4]= \beta_3 E_3 + \beta_4 E_4.$
\end{proof}

 We express the metric $g$ in some coordinates as in the following lemma.

\begin{lemma} \label{claim112bx}
Under the same hypothesis as Lemma \ref{claim112nax},

\smallskip
for each point $p_0$ in $U$, there exists a neighborhood $V$ of $p_0$ in $U$ with coordinates $(s,t, x_3, x_4)$  such that $\nabla s= \frac{\nabla f }{ |\nabla f |}$ and $g$ can be written on $V$ as
\begin{equation} \label{mtr1a}
g= ds^2 +  p(s)^2  dt^2 +    h(s)^2 \tilde{g},
\end{equation}
 where  $p:=p(s)$ and $h:=h(s)$ are smooth functions of $s$ and
 $\tilde{g}$ is (a pull-back of) a Riemannian metric of constant curvature, say $k$, on a $2$-dimensional domain with $x_3, x_4$ coordinates.

\end{lemma}

 \begin{proof}
This proof is little different from the corresponding one in Ricci soliton case. We only sketch the frame of argument and one may refer to the proof of Lemma 4.3 in \cite{Ki} for details.

We let $D^1$ be the 2-dimensional distribution spanned by $E_1 = \nabla s$ and $E_2$. And let $D^2$ be the one spanned by $E_3$ and $E_4$.   Then $D^1$ and $D^2$ are both integrable by Lemma \ref{claim112nax}.
We may consider the coordinates  $(x_1, x_2, x_3, x_4)$ from Lemma 4.2 of \cite{Ki}, so that $D^1$ is tangent to the 2-dimensional level sets

 $ \{ (x_1, x_2, x_3,x_4) |  \ x_3, x_4 \ {\rm constants} \} $ and $D^2$ is tangent to the level sets $ \{ (x_1, x_2, x_3,x_4) |  \ x_1, x_2 \ {\rm constants} \} $.
As $D_1$ and $D_2$ are orthogonal, we get the metric description for $g$ as follows;

$g= g_{11}dx_1^2 +   g_{12} dx_1 \odot dx_2  +  g_{22}dx_2^2 +    g_{33}dx_3^2 +   g_{34} dx_3 \odot dx_4 + g_{44}dx_4^2 $, where $\odot$ is the symmetric tensor product and $g_{ij}$ are functions  of $(x_1, x_2, x_3, x_4)$.

\bigskip

Using   Lemma  \ref{abc60x}, Lemma \ref{abc60byx} and   Lemma \ref{claim112nax}, one then shows that
  \begin{equation}
 g_{11}dx_1^2 +   g_{12} dx_1 \odot dx_2  +  g_{22}dx_2^2 =  ds^2 +   p(s)^2 dt^2, \hspace{2cm} \nonumber
   \end{equation}
   and
  \begin{equation}
   g_{33}dx_3^2 +   g_{34} dx_3 \odot dx_4 + g_{44}dx_4^2 = h(s)^2 \tilde{g}, \hspace{3.2cm}  \nonumber
   \end{equation}
for some functions $p=p(s)$ and $h=h(s)$ and a 2-dimensional metric $\tilde{g}$.

One can use Lemma 5.1 of \cite{Ki} to prove that
$\tilde{g}$ has constant curvature, say $k$.
\end{proof}

\section{Analysis of the metric when $\lambda_2 \neq \lambda_3 =\lambda_4$}
We continue to suppose that  $ \lambda_2  \neq \lambda_3=  \lambda_4$ for an adapted frame fields $E_j$, $j=1,2,3,4$.

The metric $ \tilde{g}$ in (\ref{mtr1a}) can be written locally; $\tilde{g} = dr^2 + u(r)^2 d \theta^2$ on a domain in $\mathbb{R}^2$ with polar coordinates $(r, \theta)$, where $u^{''}(r)=-k u$.
We set an orthonormal basis $e_3 = \frac{\partial}{\partial r}$ and $e_4 = \frac{1}{u(r)} \frac{\partial}{ \partial \theta }$.

\begin{lemma} \label{112typeb1x}
For the local metric $g = ds^2 + p(s)^2  dt^2 + h(s)^2  \tilde{g}$ with harmonic  curvature satisfying (\ref{0002bx}) with nonconstant $f$, obtained in Lemma \ref{claim112bx},  if we set $E_1 = \frac{\partial }{\partial  s}$, $E_2 = \frac{1}{p(s)}\frac{\partial }{ \partial  t}$,  $E_3 =  \frac{1}{h(s)} e_3$ and $E_4 =  \frac{1}{h(s)} e_4$, where $e_3$ and $e_4$ are  as in the above paragraph, then we have the following. Here $R_{ij} = R(E_i, E_j)$ and $R_{ijkl} = R(E_i, E_j, E_k, E_l)$.

\medskip
$\nabla_{E_1}E_i =  0$, for $i=2,3,4$,
$\ \ \ \ \ \nabla_{E_2}  E_2 = -\zeta_2  E_1$, $\ \ \ \ \nabla_{E_3}  E_3 = -\zeta_3  E_1  $,
\begin{eqnarray}  \label{ricci34}
\zeta_2&= \frac{p^{'}}{p}, \ \ \ \ \ \  \zeta_3=\zeta_4=  \frac{h^{'}}{h}  \hspace{6.8cm}   \nonumber \\
R_{1221} &= -  \frac{p^{''}}{p}=-\zeta_2^{'}
 -  \zeta_2^2, \ \ \ \  \ R_{1ii1} =-\zeta_i^{'}
 -  \zeta_i^2= -  \frac{h^{''}}{h},   \ \ {\rm  for}  \ i = 3,4. \hspace{0.3cm}   \nonumber \\
R_{11} & =  -\zeta_2^{'}
 -  \zeta_2^2 -2\zeta_3^{'}
 -  2\zeta_3^2  =  - \frac{p^{''}}{p}- 2\frac{h^{''}}{h}. \hspace{4.3cm}  \nonumber \\
R_{22} & =  -\zeta_2^{'}
 -  \zeta_2^2 -2\zeta_2 \zeta_3 = -\frac{p^{''}}{p} -2 \frac{p^{'}}{p} \frac{h^{'}}{h} . \hspace{4.8cm}  \nonumber  \\
R_{33} &=R_{44} =     -\zeta_3^{'}
 -  \zeta_3^2 -\zeta_3 \zeta_2  -(\zeta_3)^2  + \frac{k}{h^2}= -\frac{h^{''}}{h} - \frac{p^{'}}{p} \frac{h^{'}}{h} - \frac{(h^{'})^2}{h^2}  + \frac{k}{h^2}.   \nonumber
\end{eqnarray}

\end{lemma}

\begin{proof}
From the proof of Lemma \ref{claim112nax}, we already have $\nabla_{E_1}  E_2=  0$,
$\nabla_{E_3}  E_4 = \beta_3 E_3  $  and $\nabla_{E_4}  E_3 = - \beta_4 E_4  $.

As $\langle\nabla_{E_1}  E_3,  E_2\rangle = -\langle E_3,  \nabla_{E_1}  E_2\rangle=  0$, one can readily get $\nabla_{E_1}  E_3= \rho E_4$ for some function $\rho$  and $\nabla_{E_1}  E_4= -\rho E_3$.  We get $\rho=0$ by computing directly (in coordinates) $\nabla_{E_1}  E_3= \nabla_{\frac{\partial }{\partial  s}}   \frac{1}{h(s)} \frac{\partial }{\partial  r}=0$.

From Lemma \ref{abc60byx} and Lemma \ref{abc60x} (iii); $ (\lambda_2 - \lambda_i   )\langle\nabla_{E_2}  E_2, E_i\rangle =  E_i(\lambda_2) =0  $ for $i=3,4$ and
$\langle\nabla_{E_2}  E_2, E_1\rangle = - \langle  E_2,  \nabla_{E_2}  E_1 \rangle= -\zeta_2(s)$.
So, $\nabla_{E_2}  E_2 = -\zeta_2(s) E_1  $.
By similar argument,
$\nabla_{E_3}  E_3 = -\zeta_3  E_1 -  \beta_3  E_4  $, $\nabla_{E_4}  E_4 = -\zeta_4  E_1 + \beta_4 E_3  $, for some functions $\beta_3$ and $\beta_4$. Direct coordinates computation gives $\beta_3 =0$.

Then $\nabla_{E_2}  E_3=   q E_4$ for some function $q$  and $\nabla_{E_2}  E_4=  - q E_3$. One computes directly  $q=0$. We simply get $\nabla_{E_3}  E_2=0$ and $\nabla_{E_4}  E_2=0$.

We compute directly that $\nabla_{E_2}E_1 = \frac{p^{'}}{p} E_2 $ and $\nabla_{E_3}E_1 = \frac{h^{'}}{h} E_3 $ so that
 (\ref{lambda06ax}) gives $\zeta_2= \frac{p^{'}}{p}$ and $\zeta_3=\zeta_4=  \frac{h^{'}}{h}$.
We can also get $\nabla_{E_3}E_4 = 0$, $\ \ \ \ \nabla_{E_4}E_3 =  -\beta_4 E_4,  \ $ where $  \ \beta_4 = \frac{u^{'}(r)}{h(s)u(r)}$.

These computations would help to compute the curvature components.
\end{proof}

We denote $a:=\zeta_2$ and $b:=\zeta_3$.

\begin{lemma} \label{v01}
For the local metric $g = ds^2 + p(s)^2  dt^2 + h(s)^2  \tilde{g}$ with harmonic  curvature satisfying (\ref{0002bx}) with nonconstant $f$, obtained in Lemma \ref{claim112bx},
it holds that
\begin{eqnarray}
( ab +\frac{R}{12})b = 0.
\end{eqnarray}

\end{lemma}

\begin{proof} (\ref{reex})  gives
\begin{eqnarray}
  2a^{'} + 2a^2 +  2ab +\frac{R}{3}=0, \hspace{0.5cm}  \label{m01x} \\
 2b^{'} +  3b^2 +ab  - \frac{k}{h^2} + \frac{R}{3}=0. \label{m02x}
\end{eqnarray}
From
 $\nabla d f (E_i, E_i) = f (Rc- \frac{R}{3}g)(E_i, E_i) + x R(E_i, E_i) + y(R)$, we get
 $-(\nabla_{E_i} E_i)  f= f (R_{ii}- \frac{R}{3})+xR_{ii} +y(R) = -f R_{1ii1}+xR_{ii} +y(R)  $, for $i=2,3$. From  Lemma \ref{112typeb1x}  we have
\begin{eqnarray}
f^{'} a =  f (a^{'}+  a^2)-x(a^{'}+  a^2 + 2ab ) +y(R),
  \ \    \  \hspace{0.3cm}  \label{m04x} \\
f^{'} b =   f(b^{'}
+  b^2)-x(b^{'} +  2b^2+ab -\frac{k}{h^2}  ) +y(R) .  \label{m04bx}
\end{eqnarray}

From the harmonic curvature condition we have;
\begin{eqnarray} \label{m05x}
  \ \ 0 & =  \nabla_{E_1} R_{22}  - \nabla_{E_2} R_{12}  \ \   =   \   \nabla_{E_1}  (R_{22})   +  R ( \nabla_{E_{2} }  E_{1},   E_{2})   + R ( \nabla_{E_2 }E_2,   E_1) \nonumber \\
&= ( R_{22} )^{'} +a(R_{22}-  R_{11} )  \hspace{6.1cm}\nonumber \\
& =   ( -a^{'}-  a^2 -2a b )^{'}  + a ( -2ab  +2b^{'}+  2b^2)  \hspace{3.1cm}\nonumber \\
& =  -a^{''}-2aa^{'}-2a^{'} b -2a^2b +  2ab^2. \hspace{4.2cm}
\end{eqnarray}

Differentiate (\ref{m01x}) to get  $ \ \   a^{''} + 2aa^{'} +  a^{'}b+ ab^{'} =0$. Together with (\ref{m05x}) we obtain
\begin{eqnarray} \label{m07x}     ab^{'}-a^{'} b -2a^2b +  2ab^2 =0 .
\end{eqnarray}

Put (\ref{m01x}) and (\ref{m02x}) into (\ref{m07x}) to get;
\begin{eqnarray*}
0&=   -a( 3b^2 +ab  - \frac{k}{h^2} + \frac{R}{3}  )+2(a^2 +  ab +\frac{R}{6}) b -4a^2b +  4ab^2 \\
&= a \frac{k}{h^2}   + \frac{R}{3} (b-a) +3ab(b-a). \hspace{4.5cm}
\end{eqnarray*}
Then, as $a \neq b$,
\begin{eqnarray} \label{m08x}   \frac{a}{a-b} \frac{k}{h^2}  = \frac{R}{3}  +3ab.
\end{eqnarray}

\noindent From  (\ref{m04x}) and  (\ref{m04bx}) we get

 $ \frac{f^{'}}{f}(a-b)= (a^{'}
+ a^2 - b^{'}- b^2 ) - \frac{x}{f} ( a^{'}+  a^2 + 2ab - b^{'} -  2b^2-ab +\frac{k}{h^2}  )$.

\noindent  With  (\ref{m02x}) and (\ref{m01x}), the above gives

$ 2\frac{f^{'}}{f}(a-b)=(1+\frac{x}{f}) ( b^2 -ab   -\frac{k}{h^2}  ) $.

\noindent  Then by  (\ref{m08x}), $2\frac{f^{'}}{f}a =  (1+\frac{x}{f}) (-ab -   \frac{ka}{h^2(a-b)} ) =(1+\frac{x}{f})(-4ab - \frac{R}{3})$.

\medskip
\noindent Meanwhile, (\ref{m04x}) and (\ref{m01x}) gives
  $f^{'} a = - f (ab +\frac{R}{6})-x(ab -\frac{R}{6}  ) +y(R),$ so

   $2\frac{f^{'}}{f} a = -  2(ab +\frac{R}{6})-\frac{2x}{f}(ab -\frac{R}{6}  ) + \frac{2y(R)}{f}$, which should equal $(1+\frac{x}{f})(-4ab - \frac{R}{3})$.

\noindent So we obtain
\begin{eqnarray} \label{m10x}
x(ab +\frac{R}{3}  ) + y(R)= -fab.
\end{eqnarray}

\noindent  Differentiating (\ref{m10x}) and dividing by $f$,

$ \frac{f^{'}}{f} ab  = -\frac{x}{f} (a^{'}b + ab^{'})  -  (a^{'}b + ab^{'}   )   $.

\noindent From (\ref{m04x}) we get $\frac{f^{'}}{f} a b =   (a^{'}+  a^2)b - \frac{x}{f}(a^{'}+  a^2 + 2ab )b + \frac{y b}{f}$,

Equating the above two and arranging terms, we get

  $\frac{x}{f}( -ab^{'}+ a^2b + 2ab^2 )  = 2a^{'}b+ ab^{'}+ a^2b   + \frac{y b}{f}$. Use (\ref{m10x}) to get
\begin{eqnarray} \label{m28x}
\frac{x}{f}( -ab^{'}+ a^2b + 3ab^2 +\frac{R}{3}b)  =  2a^{'}b+ ab^{'}+ a^2b  -ab^2.
\end{eqnarray}

Using (\ref{m07x}) and (\ref{m01x}), the left hand side of (\ref{m28x}) equals $ \frac{x}{f}( 6ab^2 +\frac{R}{2}b)$,
while the right hand side  equals $ -(6ab^2  +\frac{R}{2}b)$.

  So we get $(1+\frac{x}{f})( 6ab +\frac{R}{2})b = 0$. Then $( ab +\frac{R}{12})b = 0$.

\end{proof}

\begin{prop} \label{abR}
For the local metric $g = ds^2 + p(s)^2  dt^2 + h(s)^2  \tilde{g}$ with harmonic  curvature satisfying (\ref{0002bx}) with nonconstant $f$, obtained in Lemma \ref{claim112bx},
suppose that  $ ab =-\frac{R}{12}$.

Then $R=0$, $y(0)=0$ and $p$ is a constant. The metric
$g$ is locally isometric to
 a domain in the non-conformally-flat static space $( W^3   \times \mathbb{R}^1, g_W + dt^2  )$ of Example 3 in Subsection 2.1.
And $f = ch^{'}(s)- x$.

\end{prop}

\begin{proof}  As $ ab =-\frac{R}{12}$, (\ref{m10x}) gives
$\frac{R}{4}x   + y(R)= \frac{R}{12}f.$

If $R \neq 0$, then $f$ is a constant, a contradiction to the hypothesis.  Therefore $R=0$.
Then $y(0)=0$ from (\ref{m10x}). From (\ref{m01x}), $a^{'} + a^2  =0$ and we have two cases: (i) $a = \frac{1}{s+ c}$ for a constant $c$ or (ii) $a= 0$.

\bigskip
Case (i); $a = \frac{1}{s+ c}$.

\noindent  From (\ref{m04x}),
$f^{'} a = 0,$ so $f$ is a constant, a contradiction to the hypothesis.

\bigskip
Case (ii); $a=0$, i.e. $p$ is a constant.

\noindent From (\ref{m04bx}) and  (\ref{m02x}), we get $f^{'} \frac{h^{'}}{h} = (f+x)\frac{h^{''}}{h}$.  If $ h^{'}$ vanishes, we get $\lambda_2  = \lambda_3$ a contradiction. So we may assume that $h$ is not constant. Then $ch^{'} =f+x $ for a constant $c \neq 0$.
Evaluating (\ref{0002bx}) at $(E_1, E_1)$,
\begin{eqnarray} \label{m20xx}
f^{''} = (f+x) R(E_1, E_1 ) -   \frac{R}{3} f  + y(R),
\end{eqnarray}
Here we get
 $f^{''} =   -2(f+x)\frac{h^{''}}{h}  $, so $h^{'''} =   -2h^{'}\frac{h^{''}}{h}    $.
Hence,  for a constant $\alpha$,
\begin{eqnarray} \label{m20x}
h^2 h^{''} =\alpha.
\end{eqnarray}

From (\ref{m02x}),
 $ 0=2b^{'} +  3b^2  - \frac{k}{h^2} = 2( \frac{h^{''}}{h}  ) +  (\frac{h^{'}}{h})^2  - \frac{k}{h^2} =
\frac{2\alpha}{h^3}   +  (\frac{h^{'}}{h})^2  - \frac{k}{h^2}.$
So we have
\begin{eqnarray} \label{m11x}
(h^{'})^2+  \frac{2\alpha}{h}  - k=0.
\end{eqnarray}

We have got exactly (\ref{m12}) and (\ref{m13}) in the case $R=0$ and $n=3$.
At this point we may write
$g = ds^2 + dt^2 + h(s)^2 \tilde{g}=   (k- \frac{2\alpha}{h})^{-1} dh^2 +   dt^2 + h(s)^2 \tilde{g}$.

\medskip
When $\alpha=0$,   $(h^{'})^2  = k \geq 0$. As $h$ is not constant,  $k> 0$. When $h^{'} = \pm \sqrt{k}  \neq 0$,  $h = \pm \sqrt{k} s + c_0$, for a constant $c_0$.
One can see that $g$ is a flat metric, a contradiction to $\lambda_2  = \lambda_3$.

\medskip
When $\alpha>0$, then $k >0$ from (\ref{m11x}). We set $r := \frac{h}{\sqrt{k}}$, and then  $g =    (1- \frac{2\alpha}{k \sqrt{k} r})^{-1} dr^2 +   dt^2 + r^2 \tilde{g}_{1}$, where $\tilde{g}_{1}$ is the metric of constant curvature $1$ on $S^2$.
When $\alpha<0$,   the 3-d metric $    (1- \frac{2\alpha}{k \sqrt{k} r})^{-1} dr^2  + r^2 \tilde{g}_{1}$ corresponds to the case III.1 of Kobayashi's. It is incomplete as explained in his proposition 2.4.

In these two cases of $\alpha >0$  and   $\alpha <0$, we get the same Riemannian metrics as those of static spaces $( W^3   \times \mathbb{R}^1, g_W + dt^2  )$ explained in Example 3. And $f = c h^{'} - x$.

Conversely, these metrics have harmonic curvature and satisfy (\ref{0002bx}) with the above $f$. Indeed, nontrivial components of (\ref{0002bx}) are  (\ref{m04x}), (\ref{m04bx}) and (\ref{m20xx}) whereas the harmonic curvature condition essentially consist of (\ref{m05x}) and the equation $ \nabla_{E_1} R_{33}  - \nabla_{E_3} R_{13}=0$; all these can be verified from $a=R=y(0)=0$ and $h, f$ which  satisfy (\ref{m20x}), (\ref{m11x}) and  $f = c h^{'} - x$.
\end{proof}

\begin{prop} \label{v02}
For the local metric $g = ds^2 + p(s)^2  dt^2 + h(s)^2  \tilde{g}$ with harmonic  curvature satisfying (\ref{0002bx}) with nonconstant $f$, obtained in Lemma \ref{claim112bx}, suppose that $b=0$ and that $ ab=0 \neq -\frac{R}{12}$.
Then the followings hold;

\medskip
{\rm (i)} $x \frac{R}{3}   + y(R)= 0.$

\medskip
{\rm (ii)} when $R>0$,  $g$ is locally isometric  to the Riemannian product $ \mathbb{S}^2(\frac{R}{6}) \times   \mathbb{S}^2(\frac{R}{3})  $ and
 $ f=c_1 \cos( \sqrt{\frac{R}{6}}s )  -x$  for any constant $c_1$, where $s$ is the distance on $\mathbb{S}^2(\frac{R}{6})$ from a point.

{\rm (iii)} when $R<0$,  $g$ is locally isometric to  $ (\mathbb{H}^2(\frac{R}{6}) \times   \mathbb{H}^2(\frac{R}{3}),    g_{\frac{R}{6}}  +  g_{\frac{R}{3}} )$, where $g_{\delta}$ is the 2-dimensional Riemannian metric of constant curvature $\delta$.
The metric $g_{\frac{R}{6}}$ can be written as  $g_{\frac{R}{6}} = ds^2 + p(s)^2 dt^2 $ where $p^{''}+ \frac{R}{6}p =0$ and then $f= c_2 p^{'}-x$ for any constant $c_2$.
\end{prop}

\begin{proof}

\bigskip
As $b = 0$, (\ref{m10x}) gives {\rm (i)}. Next, (\ref{m02x}) gives
$   \frac{k}{h^2} = \frac{R}{3}$ and  (\ref{m01x}) gives  $a^{'} + a^2  +\frac{R}{6}=\frac{p^{''}}{p}   +\frac{R}{6}  =0$.
These and (\ref{m04x})  gives
\begin{eqnarray} \label{m25x}
f^{'} a  =- \frac{R}{6} (f + x).
\end{eqnarray}

Assume $R > 0$.  Set $r_0 =   \sqrt{\frac{R}{6}} $.    For some constants $C_1\neq 0$ and $s_0$,
 $  p = C_1 \sin( r_0(s + s_0) ) $ so that $a=  r_0 \cot( r_0(s + s_0) ) $.
Then (\ref{m25x}) and  {\rm (i)} gives $f=  c_1 \cos( r_0(s +s_0) ) -x $.
Then $g= ds^2 +  \sin^2( r_0(s+s_0) ) dt^2 + \tilde{g}_{\frac{R}{3}}$ by absorbing a constant into $dt^2$ and using $   \frac{k}{h^2} = \frac{R}{3}$.

Replacing $s + s_0$ by new $s$, we have $g= ds^2 +  \sin^2( r_0s ) dt^2 + \tilde{g}_{\frac{R}{3}}$. Here $s$ becomes the distance on $\mathbb{S}^2(\frac{R}{6})$ from a point.
And $f=   c_1 \cos( r_0s ) -x $.

\medskip

Assume $R < 0$. One can argue similarly as above  and get  $g= ds^2 +  (c_3 e^{r_1 s} + c_4 e^{-r_1 s})^2 dt^2 + \tilde{g}_{\frac{R}{3}}$   and $f=  c_5 e^{r_1 s}   + c_6 e^{-r_1 s} -x $, where $r_1 = \sqrt{-\frac{R}{6}}$ and $c_3 c_6 + c_4 c_5 =0$.

 \medskip
Conversely, the above product metrics clearly have harmonic curvature. One can check they satisfy (\ref{0002bx}). Indeed, as in the proof of Proposition \ref{abR} one may check    (\ref{m04x}), (\ref{m04bx}), (\ref{m20xx}).
\end{proof}

\section{Local 4-dimensional space with harmonic curvature}

We first treat the remaining case of $\lambda_2 = \lambda_3 =\lambda_4$
and then give the proof of Theorem \ref{local}.

 \begin{prop} \label{lcf1} Let $(M,g,f)$ be a four dimensional Riemannian manifold with harmonic curvature satisfying (\ref{0002bx}) with non constant $f$. Suppose that $\lambda_2 = \lambda_3 =\lambda_4 \neq \lambda_1$ for an adapted frame field in an open subset $U$ of $M_{r} \cap \{ \nabla f \neq 0  \}$.

\smallskip
Then for each point $p_0$ in $U$, there exists a neighborhood $V$ of $p_0$ in $U$ where  $g$ is a warped product;
\begin{equation} \label{metr}
g= ds^2 +    h(s)^2 \tilde{g},
\end{equation}
where  $h$ is a positive function and the Riemannian metric $\tilde{g}$ has constant curvature, say $k$.  In particular, $g$ is conformally flat.

As a Riemannian manifold, $(M,g)$ is locally one of Kobayashi's warped product spaces, as described in the section 2 and 3 of \cite{Ko} so that
 \begin{eqnarray} \label{handf3}
 h^{''} +   \frac{R}{12}h = ah^{-3}
 \end{eqnarray}
 for a constant $a$,
 and integrating it for some constant $k$,
\begin{eqnarray} \label{handf4}
(h^{'})^2 + \frac{2a}{n-2}h^{2-n} + \frac{R}{n(n-1)}h^2 =k.
 \end{eqnarray}

And $f$ is a non-constant solution to;
\begin{eqnarray} \label{handf4xx}
h^{'}f^{'} - f h^{''}= x (h^{''} +\frac{R}{3}h)+ y(R)h.
 \end{eqnarray}

 Conversely, any $(h,f)$ satisfying (\ref{handf3}), (\ref{handf4}) and (\ref{handf4xx}) gives rise to
 $(g,f)$ which has harmonic curvature and satisfies (\ref{0002bx}).

 \end{prop}

 \begin{proof} To prove that $g$ is in the form of (\ref{metr}), we may use Lemma \ref{threesolbx} {\rm (v)} and Lemma  \ref{abc60x} {\rm (iii)}, {\rm (iv)}.
For actual proof we refer to  that of Proposition 7.1 of \cite{Ki}
since the argument is almost the same as in the gradient Ricci soliton case.
 To prove that  $\tilde{g}$ has constant curvature, we use Derdzi\'{n}ski's Lemma 4  in \cite{De}.
Then (\ref{metr}) metric is conformally flat.

\medskip
 In the setting of Lemma \ref{threesolbx}, $f$ is a function of $s$ only.
For $g= ds^2 +    h(s)^2 \tilde{g}$, in a local adated frame field, we have
\begin{eqnarray} \label{handf01}
R_{11} = - 3\frac{h^{''}}{h}, \ \ \ \ \ \  R_{ii}= -\frac{h^{''}}{h} - 2 \frac{(h^{'})^2}{h^2} + 2\frac{k}{h^2}\nonumber \\
R_{ij} =0 \ \ \ {\rm for} \ \ \ \  i \neq j \hspace{4cm} \nonumber \\
R  = -6 \frac{h^{''}}{h}  - 6\frac{(h^{'})^2}{h^2}  + 6\frac{k}{h^2} \hspace{3cm}
\end{eqnarray}

 Feeding $(E_i, E_i)$, $i=1,2$ to (\ref{0002bx}) we may get;

\begin{eqnarray}
f^{''} = - 3f\frac{h^{''}}{h}- f\frac{R}{3}  - 3x\frac{h^{''}}{h} + y(R),  \label{handfbb} \\
 h^{'}f^{'} - f h^{''}= x (h^{''} +\frac{R}{3}h)+ y(R)h.  \label{handf}
 \end{eqnarray}

Differentiating (\ref{handf}) and using  (\ref{handfbb}); we get $  (f+x) \{ h^{'''}+3  \frac{h^{''}h^{'}}{h} +\frac{R}{3}h^{'} \}=0$.  As $f \neq -x$, we get  $h^{'''}+3  \frac{h^{''}h^{'}}{h} +\frac{R}{3}h^{'}=0$.
 Multiply this by $h^3$,
 we get $(h^3 h^{''} +  \frac{R}{12}h^4)^{'} =0$. Then we have (\ref{handf3}) and then (\ref{handf4}).
Kobayashi solved these completely according to each parameter and initial condition.

 One can check that any $h$ and $f$ satisfying (\ref{handf}), (\ref{handf3}) and (\ref{handf4}) satisfy  (\ref{handf01}) and (\ref{handfbb}).
\end{proof}

\bigskip
\noindent We are ready to prove Theorem \ref{local}.

\medskip
{\bf Proof of Theorem \ref{local}.}
Recall that we have already discussed the case  $ \lambda_1=\lambda_2 = \lambda_3 =\lambda_4$ in Example 1 of Subsection 2.1.2.
The conformally flat spaces in Example 1  belong to the type {\rm (iv)} of Theorem \ref{local}; especially they have $a=0$ in (\ref{handf5}) and (\ref{handf6}).

\smallskip
As the metric $g$ and $f$ are real analytic, the Ricci-eigen values $\lambda_i$'s are real analytic on $M_{r} \cap \{ \nabla f \neq 0  \}$. And $\zeta_i$'s are real analytic from (\ref{lambda06ax}).
So we can combine Proposition \ref{4dformcx}, Lemma \ref{v01},  Proposition  \ref{abR}, \ref{v02}, \ref{lcf1} and Example 1 of Subsection 2.1.2,  to  obtain a classification of 4-dimensional {\it local} spaces with harmonic curvature satisfying (\ref{0002bx}) as Theorem \ref{local}.

 \hfill $\Box$

\medskip

\begin{remark}
{\rm In the statement of Theorem \ref{local}, among the types {\rm (i)}$\sim${\rm (iv)},  there is possibly only one type of neighborhoods $V$ on a {\it connected} space $(M,g,f)$; it holds from continuity argument of Riemannian metrics.
Then one can prove that $\tilde{M} =M$ if $M$ is of type {\rm (i)}, {\rm (ii)} or {\rm (iii)}.}

\end{remark}

\section{complete 4-dimensional space with harmonic curvature}

It is not hard to describe complete spaces corresponding to {\rm (i),(ii), (iii)} of Theorem \ref{local}.

For complete conformally flat case corresponding to {\rm (iv)} of Theorem \ref{local}, we may use Theorem 3.1 of Kobayashi's classification \cite{Ko}. Then $(M,g)$  can be either $\mathbb{S}^4$, $\mathbb{H}^4$, a flat space or one of the spaces in  Example 1$\sim$5 in \cite{Ko}.
Now our task is to determine $f$, which is described by (\ref{handf4xxx}).

\medskip
We first recall the spaces in Example 3$\sim$5 in \cite{Ko}.
Any space in Example 3 and 4 in \cite{Ko} is a quotient of a warped product $\mathbb{R} \times_h N(1)$ where $h$ is a smooth periodic function on $\mathbb{R}$; recall that $N(k)$ is a Riemannian manifold of constant sectional curvature $k$. Any space in Example 5 in \cite{Ko} is a quotient of a  warped product $\mathbb{R} \times_h N(k)$ where $h$ is smooth on $\mathbb{R}$.  Here $h \geq \rho_1 >0$.

We verify the following lemma.

\begin{lemma} \label{fglo}
For any one of the spaces in Example 3, 4 and 5 in \cite{Ko},  the following holds.

\medskip
{\rm (i)} The solution $f$ to (\ref{0002bx}) can be defined and smooth  on $\mathbb{R}$.

{\rm (ii)} If $h$ is periodic and $ x\frac{R}{3}+ y(R)=0$, then  $f$ is periodic.
\end{lemma}

\begin{proof} As stated in Theorem \ref{lcf1}, any $(h,f)$ satisfying (\ref{handf3}), (\ref{handf4}) and (\ref{handf4xx}) gives rise to
 $(g,f)$ which satisfies (\ref{0002bx}). So, $(h,f)$ satisfies (\ref{handfbb}).

Choose some point $s_0$ with $h^{''}(s_0) \neq 0$. We
consider  the ODE for any constant $c$,
 \begin{eqnarray} \label{fR12}
 f^{''}  =-f( \frac{R}{12} +3ah^{-4})  +3x( \frac{R}{12} -ah^{-4}) + y(R),
 \end{eqnarray}
 with initial conditions $f^{'}(s_0) = c$ and $f(s_0) =  \frac{c h^{'}(s_0)  -\{ x (h^{''}(s_0) +\frac{R}{3}h(s_0))+ y(R)h(s_0) \}}{ h^{''}(s_0) }$ so that
 (\ref{handf4xxx}) holds at $s_0$.
 Note that (\ref{fR12}) is equivalent to (\ref{handfbb}) since $h$ satisfies (\ref{handf5}).

 As $h$ exists smoothly on $\mathbb{R}$ as a solution of (\ref{handf5}),
by global Lipschtz continuity of the right hand side of (\ref{fR12}), the solution $f$ exists globally on $\mathbb{R}$.

We get from (\ref{handf5});
 \begin{eqnarray} \label{fR12b}
 h^{'''}   = - (\frac{R}{12}  +3ah^{-4}) h^{'}.
 \end{eqnarray}

 Then by (\ref{fR12}) and (\ref{fR12b}) it satisfies
$  h^{'}f^{''}  - f h^{'''}    = x (h^{'''} +\frac{R}{3}h^{'})+ y(R)h^{'} $, which is the derivative of (\ref{handf4xxx}). So, (\ref{handf4xxx}) holds on $\mathbb{R}$.   As $h$ and $f$ satisfy (\ref{handf4xxx}),  the induced $(g,f)$ satisfies (\ref{0002bx}) on $\mathbb{R}$.

\medskip
If $ x\frac{R}{3}+ y(R)=0$, then from (\ref{handf4xxx}) we get $ f(s) =  -x+ C h^{'}(s)$ for a constant $C$, which is periodic as $h$.
\end{proof}
About Lemma \ref{fglo} {\rm (ii)}, we note that if  $ x\frac{R}{3}+ y(R)\neq 0$ and $h$ is periodic, then the periodicity of $f$ should be checked by computation.

 We are ready to state;

\begin{thm} \label{complete}

Let $(M, g)$ be a four dimensional complete Riemannian manifold with harmonic curvature, satisfying (\ref{0002bx}) with non-constant $f$.
Then it is one of the following;

\medskip

{\rm (\ref{complete}-i)}
 $(M,g)$ is isometric  to a quotient of    $ \mathbb{S}^2(\frac{R}{6}) \times   \mathbb{S}^2(\frac{R}{3})$ with $R>0$, where
And  $ f=c_1 \cos( \sqrt{\frac{R}{6}}s )  -x$  for any constant $c_1$, where $s$ is the distance on $\mathbb{S}^2(\frac{R}{6})$ from a point.  It holds that $x \frac{R}{3}   + y(R)= 0.$

\medskip
{\rm (\ref{complete}-ii)} $(M,g)$ is isometric  to a quotient of   $ (\mathbb{H}^2(\frac{R}{6}) \times   \mathbb{H}^2(\frac{R}{3}),    g_{\frac{R}{6}}  +  g_{\frac{R}{3}} )$ with $R<0$.
And $f= c_2 \cosh(s)-x$ for any constant $c_2$, where $s$ is the distance function on $ \mathbb{H}^2(\frac{R}{6})$ from a point. It holds that $x \frac{R}{3}   + y(R)= 0.$

\medskip

{\rm (\ref{complete}-iii)} $(M,g)$ is isometric  to a quotient of   one of the static spaces in Example 3 of Subsection 2.1.2, which is the Riemannian product $(\mathbb{R}^1  \times W^3   , dt^2  + ds^2 + h(s)^2 \tilde{g})$ of $\mathbb{R}^1$  and some 3-dimensional conformally flat static space $(W^3=\mathbb{R}^1 \times   \mathbb{S}^2(1)  , ds^2 + h(s)^2 \tilde{g})$ with zero scalar curvature,
which contains the space section of the Schwarzshild space-time

And $f = c \cdot h^{'}(s)- x$ for a constant $c$. It holds that $R=y(0)=0$.

\medskip
{\rm (\ref{complete}-iv)} $(M,g)$ is conformally flat. It is either $\mathbb{S}^4$, $\mathbb{H}^4$, a flat space or one of the spaces in  Example 1$\sim$5 in \cite{Ko}. Below we describe $f$ in each subcase.

\medskip

{\rm (\ref{complete}-iv-1)}  $\mathbb{S}^4(k^2)$ with the metric  $g= ds^2 + \frac{{\sin(ks)}^2}{k^2} g_{1}$, for any constant $c$,
\begin{eqnarray*} \label{fhr016}
  f(s) =  c \cdot \cos(ks) +3x+ \frac{y(12k^2)}{k^2}.
\end{eqnarray*}

\medskip
\indent {\rm (\ref{complete}-iv-2)}  $\mathbb{H}^4(-k^2)$, with $g= ds^2 + \frac{{\sinh(ks)}^2}{k^2} g_{1}$, for any constant $c$,
\begin{eqnarray*} \label{fhr017}
  f(s) = c \cdot \cosh(ks)  +3x- \frac{y(-12k^2)}{k^2}.
\end{eqnarray*}

\indent {\rm (\ref{complete}-iv-3)} A flat space,   $f = a+\sum_i  + b_i x_i + \frac{y(0)}{2} x_i^2$ in a local Euclidean coordinates $x_i$, for constants $a$ and $b_i$.

\medskip

{\rm (\ref{complete}-iv-4)} Example 1 and 2 in \cite{Ko}; the Riemannian product $(\mathbb{R} \times N(k), ds^2 + g_{k})$ or its quotient, $k \neq 0$, where $N(k)$ is 3-dimensional complete space of constant sectional curvature $k$,

 $ \ \ \ \ \ f = c_1 \sin{\sqrt{\frac{R}{3}}}s   + c_2\cos{\sqrt{\frac{R}{3}}}s -x$ when $R>0$, or

   $ \ \ \ \ \ f = c_1  \sinh{\sqrt{-\frac{R}{3}}}s +c_2 \cosh{\sqrt{-\frac{R}{3}}}s   -x$ when $R<0$.

    It holds that  $ \  x \frac{R}{3}+ y(R)=0$  and $R  = 6k$.

\medskip
{\rm (\ref{complete}-iv-5)} Example 3 and 4 in \cite{Ko};
a  warped product $\mathbb{R} \times_h N(1)$ or its quotient, where $h$ is a periodic function on $\mathbb{R}$,
$ f$ is on $\mathbb{R}$, satisfying  {\rm (\ref{handf4xxx})}.

\medskip
{\rm (\ref{complete}-iv-6)} Example 5 in \cite{Ko}; a warped product $\mathbb{R} \times_h N(k)$ where $h$ is defined on $\mathbb{R}$,
$ f$ is  on $\mathbb{R}$, satisfying  {\rm (\ref{handf4xxx})}.

\end{thm}

\begin{proof}
To obtain {\rm (\ref{complete}-i)}, {\rm (\ref{complete}-ii)} and {\rm (\ref{complete}-iii)},
we use continuity argument of Riemannian metrics from Theorem \ref{local}.
To describe $f$ in the subcases of {\rm (\ref{complete}-iv)}, we use (\ref{handf4xxx})
and (\ref{handfbb}).
\end{proof}

\section{4-d Static spaces with harmonic curvature}

In this section we study static spaces, i.e. those satisfying (\ref{0002bx1}). As explained in the introduction, to study local static spaces is interesting due to Corvino's local deformation theory of scalar curvature. Here we state local classification which can be read off from Theorem \ref{local};

\begin{thm} \label{locals}
Let $(M, g, f)$ be a four dimensional  (not necessarily complete) static space with harmonic curvature and non-constant $f$.
Then for each point $p$ in some open dense subset $\tilde{M}$ of $M$, there exists a neighborhood $V$ of $p$ with one of the following properties;

\medskip

{\rm (\ref{locals}-i)}
 $(V,g)$ is isometric to a domain in   $ \mathbb{S}^2(\frac{R}{6}) \times   \mathbb{S}^2(\frac{R}{3})$ with $R>0$.
And $ f=c_1 \cos( \sqrt{\frac{R}{6}}(s+ s_0) )  $, where $s$ is the distance function on $ \mathbb{S}^2(\frac{R}{6})$ from a point and $c_1, s_0$ are constants.

\smallskip
{\rm (\ref{locals}-ii)} $(V,g)$ is isometric to a domain in $ (\mathbb{H}^2(\frac{R}{6}) \times   \mathbb{H}^2(\frac{R}{3}),    g_{\frac{R}{6}}  +  g_{\frac{R}{3}} )$ with $R<0$.
If we express $g_{\frac{R}{6}}$ as  $g_{\frac{R}{6}} = ds^2 + p(s)^2 dt^2 $ with $p^{''}+ \frac{R}{6}p =0$, then $f= c_2 p^{'}$ for any constant $c_2$.

\smallskip

{\rm (\ref{locals}-iii)} $(V,g)$ is isometric to a domain in one of the static spaces in Example 3 of Subsection 2.1.2, which is the Riemannian product $\mathbb{R}^1 \times W^3$ of $\mathbb{R}^1$ and some 3-dimensional conformally flat static space $(W^3, ds^2 + h(s)^2 \tilde{g})$ with zero scalar curvature. And $f = c h^{'}$.

\smallskip

{\rm (\ref{locals}-iv)} $(V,g)$ is conformally flat. So, it is one of the warped product metrics of the form  $ds^2 + h(s)^2 g_k$ whose existence is described  in the section 2 of \cite{Ko}. The function $h$ satisfies (\ref{handf5}) and (\ref{handf6}), and we have
$ f(s) = C h^{'}(s)$.

\end{thm}

For complete conformally flat case corresponding to {\rm (\ref{locals}-iv)} in Theorem \ref{locals}, if we use Theorem 3.1 of Kobayashi's classification, we get either $\mathbb{S}^4$, $\mathbb{H}^4$, a flat space or one of the spaces in  Example 1$\sim$5 in \cite{Ko}. We may obtain classification of complete four dimensional  static spaces with harmonic curvature;

\begin{thm} \label{completes}
Let $(M, g, f)$ be a complete four dimensional  static space with harmonic curvature.
Then it is one of the following;

\medskip

{\rm (\ref{completes}-i)}
 $(M,g)$ is isometric  to a quotient of  $ \mathbb{S}^2(\frac{R}{6}) \times   \mathbb{S}^2(\frac{R}{3})  $  with $R>0$. And
 $ f=c_1 \cos( \sqrt{\frac{R}{6}}s )  $, where $s$ is the distance function on $ \mathbb{S}^2(\frac{R}{6})$ from a point.

{\rm (\ref{completes}-ii)}
 $(M,g)$ is isometric  to a quotient of  $ \mathbb{H}^2(\frac{R}{6}) \times   \mathbb{H}^2(\frac{R}{3})  $ with $R<0$. And

 $f=c_2 \cosh( \sqrt{\frac{-R}{6}}s )$, where $s$ is the distance function on $ \mathbb{H}^2(\frac{R}{6})$ from a point.

\smallskip
{\rm (\ref{completes}-iii)} $(M,g)$ is isometric to a quotient of the Riemannian product $( \mathbb{R}^1  \times W^3 , \   dt^2 + \tilde{g})$, where  $(W^3, \tilde{g}) $ denotes the warped product manifold on the smooth product $\mathbb{R}^1 \times   \mathbb{S}^2(1) $ which contains the space section of the Schwarzshild space-time; see Example 3 of Subsection 2.1.2.

\smallskip
{\rm (\ref{completes}-iv)} $(M,g,f)$ is  $\mathbb{S}^4$, $\mathbb{H}^4$, a flat space or one of the spaces in  Example 1$\sim$5 in \cite{Ki}.

{\rm (\ref{completes}-v)} $g$ is a complete Ricci-flat metric with $f$ a constant function.
\end{thm}

\begin{proof}
It follows from Theorem \ref{complete}. When $f$ is nonzero constant, $g$ is clearly Ricci-flat. So we get {\rm (v)}.
\end{proof}

\medskip
Fischer-Marsden \cite{FM2} made the conjecture that any closed  static space is Einstein. But it was disproved by conformally flat examples
in \cite{La, Ko}. Now we ask

{\bf Question 1}: Does there exist a closed static space which does not have harmonic curvature?

\medskip
The space in {\rm (\ref{completes}-iii)} of Theorem \ref{completes} has  three distinct Ricci-eigenvalues.
We only know examples of static spaces with at most three distinct Ricci-eigenvalues. So, we ask;

\medskip
{\bf Question 2}: Does there exist a static space with more than three distinct Ricci-eigenvalues?
Is there a limit on the number of  distinct Ricci-eigenvalues for a static space?

\section{Miao-Tam critical metrics and $V$-critical spaces}

In this section we treat Miao-Tam critical metrics.
These metrics originate from \cite{MT2} where Miao and Tam studied
the critical points of the volume functional on the space $\mathcal{M}_{\gamma}^{K}$ of  metrics with constant scalar curvature $K$ on a compact manifold $M$ with a prescribed metric $\gamma$ at the boundary of $M$.
Miao-Tam critical metrics are precisely described \cite{MT} in case they are Einstein or conformally flat.

Here we first describe 4-d metrics with harmonic curvature which have a nonzero solution $f$ to (\ref{0002bx3}).
We do not assume the condition  $f_{| \Sigma} =0$ but still can show that any such metric must be conformally flat;

\begin{thm} \label{localv}
Let $(M, g)$ be a four dimensional  (not necessarily complete) Riemannian manifold with harmonic curvature, satisfying {\rm (\ref{0002bx3})} with non-constant $f$.
Then $(M,g)$ is conformally flat. It is one of the warped product metrics of the form  $ds^2 + h(s)^2 g_k$ whose existence is described  in the section 2 of \cite{Ko}. The function $h$ satisfies (\ref{handf5}) and (\ref{handf6}), and $f$ satisfies
$ h^{'}f^{'} - f h^{''}= - \frac{h}{n-1}.$

\end{thm}

\begin{proof}
The proof is immediate from Theorem \ref{local};  the cases (i)-(ii) of  Theorem \ref{local} require
$x \frac{R}{3} + y(R) =0$ and (iii) requires $ y(0) =0$, which contradict to the condition  $x=0$ and $y(R) = -\frac{1}{3}$ that (\ref{0002bx3}) has. The description of Theorem \ref{local}
(iv) holds for $g$ and $f$ of Theorem \ref{localv}, and in particular $g$ is conformally flat.
\end{proof}
Theorem \ref{localv} shows an advantage of our local approach over \cite{BDR} in analyzing (\ref{0002bx3}). In fact,
the integration argument of \cite[Lemma 5]{BDR} only works for compact manifolds, but our analysis can resolve local solutions.

From Theorem \ref{locals} and \ref{localv} we can classify local 4-d $V$-static spaces with harmonic curvature;

\begin{thm} \label{localsxx}
Let $(M, g, f)$ be a four dimensional  (not necessarily complete) $V$-static space with harmonic curvature and non-constant $f$.
Then for each point $p$ in some open dense subset $\tilde{M}$ of $M$, there exists a neighborhood $V$ of $p$ with one of the following properties;

\medskip

{\rm (\ref{localsxx}-i)}
 $(V,g)$ is isometric to a domain in   $ \mathbb{S}^2(\frac{R}{6}) \times   \mathbb{S}^2(\frac{R}{3})$ with $R>0$.
And $ f=c_1 \cos( \sqrt{\frac{R}{6}}(s+ s_0) )  $, where $s$ is the distance function on $ \mathbb{S}^2(\frac{R}{6})$ from a point and $c_1, s_0$ are constants.

\smallskip
{\rm (\ref{localsxx}-ii)} $(V,g)$ is isometric to a domain in $ (\mathbb{H}^2(\frac{R}{6}) \times   \mathbb{H}^2(\frac{R}{3}),    g_{\frac{R}{6}}  +  g_{\frac{R}{3}} )$ with $R<0$.
If we express $g_{\frac{R}{6}}$ as  $g_{\frac{R}{6}} = ds^2 + p(s)^2 dt^2 $ with $p^{''}+ \frac{R}{6}p =0$, then $f= c_2 p^{'}$ for any constant $c_2$.

\smallskip
{\rm (\ref{localsxx}-iii)} $(V,g)$ is isometric to a domain in one of the static spaces in Example 3 of Subsection 2.1.2, which is the Riemannian product $\mathbb{R}^1 \times W^3$ of $\mathbb{R}^1$ and some 3-dimensional conformally flat static space $(W^3, ds^2 + h(s)^2 \tilde{g})$ with zero scalar curvature. And $f = c h^{'}$ for any constant $c$.

\smallskip

{\rm (\ref{localsxx}-iv)} $(V,g)$ is conformally flat. It is one of the warped product metrics of the form  $ds^2 + h(s)^2 g_k$ whose existence is described  in the section 2 of \cite{Ko}. The function $h$ satisfies (\ref{handf5}) and (\ref{handf6}), and we have
$ f(s) = c h^{'}(s)$ for any constant $c$.

\smallskip

{\rm (\ref{localsxx}-v)} $(V,g)$ is conformally flat. It is one of the warped product metrics of the form  $ds^2 + h(s)^2 g_k$ whose existence is described  in the section 2 of \cite{Ko}. The function $h$ satisfies (\ref{handf5}) and (\ref{handf6}) and $f$ is any constant multiple of a solution $f_0$ satisfying $ h^{'}f_0^{'} - f_0 h^{''}= - \frac{h}{n-1}.$
\end{thm}
Note that the last equation in {\rm (\ref{localsxx}-v)} comes from (\ref{0002bx3f}), which allows any constant multiple of one solution.

\bigskip
As a Corollary of Theorem \ref{localv},
we could state an extension of  Miao-Tam's theorem 1.2 in \cite{MT} to the case of harmonic curvature. Instead we choose to state
the following version, which is a twin to the corollary 1 of \cite{BDR}.

\begin{thm} \label{compbdry}
If $(M^4, g, f )$ is a simply connected, compact Miao-Tam critical metric of harmonic curvature with boundary isometric to a standard sphere $S^3$. Then $(M^4, g)$ is isometric to a geodesic ball in a simply connected space form $\mathbb{R}^4, \mathbb{H}^4$ or $\mathbb{S}^4$.
\end{thm}

One can also make classification statements of complete spaces with harmonic curvature satisfying (\ref{0002bx3}) or (\ref{0002bx3f}). We omit them.

\medskip
Theorem \ref{localv} gives a speculation that it might hold in general dimension.
So, we ask;

\medskip
{\bf Question 3}: Let $(M, g)$ be an $n$-dimensional Miao-Tam critical metric with harmonic curvature. Is it conformally flat?

\bigskip
It is also interesting to find examples of non-conformally flat Miao-Tam critical metric in any dimension.

\section{On critical point metrics}

In this section we study a critical point metric, i.e. a Riemannian metric $g$ on a manifold $M$ which admits a non-zero solution $f$ to  (\ref{0002bx2}).
  According to \cite{HCY}, these critical point metrics with harmonic curvature on closed manifolds in any dimension are Einstein.

On a closed manifold, by taking the trace of this equation,  $R$ must be positive and $f$ satisfies $\int_M f \  dv=0$.
Here $M$ is not necessarily closed and $g$ may have non-positive scalar curvature.
From Theorem \ref{local}, we can easily obtain the next theorem.

\begin{thm} \label{localc}
Let $(M, g)$ be a four dimensional  (not necessarily complete) Riemannian manifold with harmonic curvature, satisfying (\ref{0002bx2}) with non-constant $f$. Then one of the following holds.

\medskip
{\rm (\ref{localc}-i)}  $(M,g)$ is locally isometric to a domain in one of the static spaces of Example 3 in Subsection 2.1.2 below, which is the Riemannian product $(\mathbb{R}^1  \times W^3   , dt^2  +  ds^2 + h(s)^2 \tilde{g})$ of $\mathbb{R}^1$  and a 3-dimensional conformally flat static space $(W^3, ds^2 + h(s)^2 \tilde{g})$ with zero scalar curvature. And $f = c \cdot h^{'}(s)- 1$.

\medskip
{\rm (\ref{localc}-ii)} $(M,g)$ is conformally flat and is locally one of the metrics whose existence is described  in the section 2 of \cite{Ko}; $g= ds^2 + h(s)^2 g_k$ where
$h$ and $f$ satisfy (\ref{handf5}), (\ref{handf6}) and (\ref{handf4xxx}).

\end{thm}

\begin{proof}
 We have $x \frac{R}{3}   + y(R)= 0$ and $R \neq 0$ in the  cases {\rm (i), (ii)} of Theorem \ref{local}. This is not compatible with (\ref{0002bx2}).
\end{proof}

Complete spaces with harmonic curvature which admit a solution $f$ to  (\ref{0002bx2})  are described in the next theorem. We obtain non-conformally-flat examples with zero scalar curvature in  {\rm (\ref{completec}-i)}, which is in contrast to the above result of  \cite{HCY} for closed manifolds. The case {\rm (\ref{completec}-v)} is also noteworthy;
it is conformally flat with positive scalar curvature and the metric $g$ can exist on a compact quotient but the function $f$
can survive on the universal cover $\mathbb{R} \times_h N(1)$.

\begin{thm} \label{completec}

Let $(M, g)$ be a four dimensional complete Riemannian manifold with harmonic curvature, satisfying (\ref{0002bx2}) with non-constant $f$. Then $(M,g)$ is one of the following;

\medskip
{\rm (\ref{completec}-i)}  $(M,g)$ is isometric  to a quotient of   one of the static spaces of Example 3 in Subsection 2.1.2 below, which is the Riemannian product $(\mathbb{R}^1  \times W^3   , dt^2  +  ds^2 + h(s)^2 \tilde{g})$ of $\mathbb{R}^1$  and a 3-dimensional conformally flat static space $(W^3, ds^2 + h(s)^2 \tilde{g})$ with zero scalar curvature
which contains the space section of the Schwarzshild space-time.
And $f = c \cdot h^{'}(s)- 1$ for a constant $c$.

{\rm (\ref{completec}-ii)} $\mathbb{S}^4(k^2)$ with the metric  $g= ds^2 + \frac{{\sin^2(ks)}}{k^2} g_{1}$, with  $ f(s) = c \cdot \cos(ks).$

\indent {\rm (\ref{completec}-iii)} $\mathbb{H}^4(-k^2)$, with $g= ds^2 + \frac{{\sinh(ks)}^2}{k^2} g_{1}$, with $f(s) =  c \cdot  \cosh(ks).$

{\rm (\ref{completec}-iv)} A flat space,   $f = a+ \sum_i  b_i x_i $ in a local Euclidean coordinates $x_i$ and constants $a, b_i$.

\smallskip
{\rm (\ref{completec}-v)}  Example 3 in \cite{Ko};
A  warped product $\mathbb{R} \times_h N(1)$ where $h$ is a periodic function on $\mathbb{R}$,
$ f$ is smooth on $\mathbb{R}$ but is not periodic. Here $R>0$.

{\rm (\ref{completec}-vi)} Example 5 in \cite{Ko}; A warped product $\mathbb{R} \times_h N(k)$ where $h$ is defined on $\mathbb{R}$,
$ f$ is is smooth on $\mathbb{R}$. Here $R \leq 0$.

\end{thm}

\begin{proof} We may check the list in Theorem \ref{complete}. The spaces of {\rm (\ref{complete}-i)} and {\rm (\ref{complete}-ii)} in Theorem \ref{complete} are excluded as in the proof of Theorem \ref{localc}.
The space for {\rm (\ref{complete}-iv-4)} of Theorem \ref{complete},  where $R \neq 0$, dose not satisfy the equation $h^{'}f^{'} - f h^{''}= x (h^{''} +\frac{R}{3}h)+ y(R)h $; when $x=1$, $y(R)=-\frac{R}{4}$  and $h=1$, it reduces to $0=\frac{R}{12}$.

On the space of {\rm (\ref{complete}-iv-5)}  in Theorem \ref{complete},  $f$ is defined and smooth on $\mathbb{R}$ by Lemma \ref{fglo} (i).
As $ x\frac{R}{3} + y(R) \neq 0 $,  Lemma \ref{fglo} (ii) does not apply. According to the section E.2 of \cite{La}, $f$ cannot be periodic. This yields {\rm (\ref{completec}-v)}.

\end{proof}

\end{document}